\documentclass[reqno]{amsart}
\usepackage{amssymb}
\usepackage{color}
\newcommand{\eoe}{~$\scriptscriptstyle\blacksquare$}

\usepackage[pagebackref]{hyperref}
\renewcommand*{\backrefalt}[4]{\ifcase #1 (\tt not cited)\or (\tt cited on page~#2)\else (\tt cited on pages~#2)\fi}

\theoremstyle{plain}
\newtheorem{theorem}{Theorem}[section]
\newtheorem{lemma}[theorem]{Lemma}
\newtheorem{proposition}[theorem]{Proposition}
\newtheorem{corollary}[theorem]{Corollary}
\newtheorem{remark}[theorem]{Remark}

\theoremstyle{definition}
\newtheorem{example}[theorem]{Example}
\newtheorem{question}[theorem]{Question}

\numberwithin{equation}{section}

\newcommand{\N}{\mathbb{N}}
\newcommand{\Z}{\mathbb{Z}}
\newcommand{\Q}{\mathbb{Q}}
\newcommand{\C}{\mathbb{C}}
\newcommand{\oQp}{\overline{\Q_p}}
\newcommand{\oZp}{\overline{\Z_p}}
\newcommand{\K}{\widehat{K}}
\newcommand{\aK}{\overline{K}}
\newcommand{\aV}{\overline{V}}
\newcommand{\oK}{\overline{\widehat{K}}}
\newcommand{\oV}{\overline{\widehat{V}}}
\newcommand{\ok}{\overline{k_v}}
\newcommand{\bK}{\mathbb{K}}          
\newcommand{\bO}{\mathbb{O}}         
\newcommand{\Gal}{\operatorname{Gal}}
\newcommand{\Aut}{\operatorname{Aut}}
\newcommand{\ratrk}{\operatorname{rat.rk}}
\usepackage[notcite, notref]{showkeys}

\newcommand{\br}{\text{br}}
\newcommand{\ur}{\text{ur}}

\usepackage{orcidlink}

\begin{document}

\title[Approximating DVRs]
{Approximating DVRs by elements of bounded ramification}

\author[G. W. Chang]{Gyu Whan Chang}
\address{Department of Mathematics Education, Incheon National University, Incheon 22012, Republic of Korea}
\email{whan@inu.ac.kr}

\author[G. Peruginelli]{Giulio Peruginelli\,\orcidlink{0000-0001-7694-8920} }
\address{Department of Mathematics ``Tullio Levi-Civita'' University of Padova, Via Trieste, 63 35121 Padova, Italy.} 
\email{gperugin@math.unipd.it}

\date{\today}

\subjclass[2020]{12J20, 13F30, 13A18, 13F05, 13J10}

\keywords{Discrete valuation domain, rational function field,
residually algebraic extension, torsion extension, perfect residue field, minimal pair, ramification index}

\begin{abstract}
Let $V$ be a DVR with quotient field $K$ and perfect residue field,  $v$ be the valuation on $K$ associated with $V$, $\widehat K$ be the completion of $K$, and $\mathbb{K}$ be the completion of an algebraic closure $\overline{\widehat{K}}$ of $\widehat K$. We show that a DVR of the rational function field $K(X)$ which is a  residually algebraic extension of $V$ is necessarily of the form $V_{\alpha}=\{\phi\in K(X)\mid v(\phi(\alpha))\geq0\}$,  for an element $\alpha$ of $\mathbb{K}$ transcendental over $K$, and that $\alpha$ is algebraic over $\widehat K$ if and only if the residue field extension is finite. Not every such $V_{\alpha}$ is a DVR, however, and we characterize the $\alpha \in\mathbb{K}$ for which $V_{\alpha}$ is a DVR: they are the elements which can be approximated by algebraic elements in $\overline{\widehat{K}}$ with bounded ramification indexes. Combining the two results, we obtain a complete description of the extensions of $V$ to $K(X)$ which are DVRs and residually algebraic over $V$, together with a criterion for each of the two cases to occur. The proofs rest on a bound for the ramification index in a compositum, valid with no tameness assumption and under a separability hypothesis on one residue field extension only; we show that the inequality cannot be improved to a divisibility and that this hypothesis cannot be dropped. We also show that the hypothesis of discreteness cannot be omitted. Furthermore, we show that the set of $\alpha\in\mathbb K$ for which $V_{\alpha}$ is a DVR is a subfield of $\mathbb K$,
which sits properly between $\overline{\widehat{K}}$ and $\mathbb K$, and corresponds to those elements $\alpha$  for which the value group of $\widehat K(\alpha)$ is discrete.
\end{abstract}

\maketitle

\section{Introduction}\label{sec:intro}

All rings considered in this paper are commutative with identity. Let $V$ be a
valuation domain with quotient field $K$ and $X$ be an indeterminate. An
\emph{extension of $V$ to the rational function field $K(X)$} is a valuation
domain $W$ of $K(X)$ with $W\cap K=V$; writing $(W,\mathfrak m)$ and
$(V,\mathfrak p)$, where $\mathfrak m$ and $\mathfrak p$ are the maximal ideals of $W$ and $V$
respectively, the extension is said to be \emph{residually algebraic} if
$W/\mathfrak m$ is algebraic over $V/\mathfrak p$ and \emph{residually
transcendental} otherwise, and it is said to be \emph{torsion} if the group
$\Gamma_w/\Gamma_v$ of its value groups is a torsion group. The extensions
that are residually algebraic and torsion are known in the literature as the
\emph{residually algebraic torsion} extensions of $V$ (see for example \cite{APZAll}).

We fix here the notation used throughout the paper, which we recall in detail in Section
\ref{sec:notation}: $V$ has rank one, $v$ is the valuation associated with $V$ and $\Gamma_v$ its value group,
$\K$ is the $v$-adic completion of $K$, $\oK$ is an algebraic closure of $\K$ and $\bK$ is the
completion of $\oK$; $\oV$ and $\bO$ are the valuation domains of $\oK$ and of $\bK$, and
$G=\Aut(\oK\mid\K)$.
The residually transcendental extensions have been characterized by Alexandru, Popescu and Zaharescu \cite{apz88} 
and described by means of \emph{minimal pairs} $(\alpha,\delta)$, with $\alpha\in\oK$ and $\delta$ 
in the divisible hull of $\Gamma_v$ \cite{apz90}; see also \cite{APZAll} for a description of all the valuations of $K(X)$ extending $v$. 
In the language of monomial valuations, which is the one we use
and that we will recall in section \ref{sec:monomial}, such an extension is equal to $\oV_{\alpha,\delta}\cap K(X)$, 
and the pair $(\alpha,\delta)$ corresponds to the closed ball $B(\alpha,\delta)=\{x\in\bK\mid v(x-\alpha)\geq\delta)\}$ of center $\alpha$ and radius $\delta$.
 The residually algebraic torsion extensions have been investigated by \"Oke \cite{oke13}, 
also in several indeterminates, and the completions of the associated valued fields of rational functions by Iovi\c{t}\u{a} and Zaharescu \cite{IZ95}, 
who identify them with the complete intermediate fields of $\bK\mid\K$ (an identification which, as we point out in Remark \ref{rem:IZ}, 
requires in one of its two directions a hypothesis that is not stated there); 
minimal pairs have since been recast in the language of key polynomials and of pseudo-convergent sequences, 
see, for instance, \cite{NS18} and \cite{Dut21}; all the extensions of $V$ to $K(X)$ have been described 
in terms of pseudo-monotone sequences, under the assumption that $\K$ is algebraically closed, by Peruginelli and Spirito \cite{PS21}, 
who also studied in \cite{PS20} the space of the valuation domains attached to pseudo-convergent sequences.

In this paper $V$ is a DVR (i.e., a discrete valuation domain of rank one) with perfect residue field, 
and we determine the DVRs of $K(X)$ which are residually algebraic extension of  $V$, 
proving that each such extension is of the form
$$V_{\alpha}=\{\phi\in K(X)\mid v(\phi(\alpha))\geq0\}$$
for an element $\alpha$ of $\bK$ transcendental over $K$, and $\alpha$ lies in
$\oK$ itself precisely when the residue field extension is finite
(Theorem \ref{thm:main}). Since every DVR of $K(X)$ lying over $V$ is
automatically a torsion extension of $V$, the class so described is exactly the
class of the discrete residually algebraic torsion extensions of $V$, 
that for short we call discrete residually algebraic extensions.
When $\alpha\in\oK$, these extensions are, in the language of the previous
paragraph, the limiting case $\delta=\infty$, in which the ball $B(\alpha,\delta)$ shrinks to a point; for
$\alpha\in\bK\setminus\oK$, which is the case the theorem is really about and in which the monomial
construction is not available.

Each of the two cases of Theorem \ref{thm:main} occurs, namely, 
either the extension of the residue fields is finite or it is infinite,
and one can say exactly when each of
them does: the first one corresponds to  the case $\alpha\in\oK$  
and arises precisely when $K$ is not complete, and the second case corresponds to
$\alpha\notin\oK$ and it occurs precisely when the residue field of $V$ admits finite extensions of
unbounded degree (Propositions \ref{rem:byproduct} and \ref{prop:nonvacuous}). Both
conditions can fail at once, and then the class is empty: this happens when $K$ is complete
and its residue field is algebraically closed or real closed (Corollary
\ref{cor:vacuous}). Both alternatives are illustrated in \S\ref{subsec:families} by two
families of examples, the power series rings $F[[t]]$ and the localizations
$F[t]_{(f)}$, the completion of the second being an instance of the first, so
that the two differ exactly by completeness; $\C[[t]]$ carries no such extension
at all, whereas $\C[t]_{(t)}$ carries infinitely many.

The proofs rest on a bound for the ramification index in a compositum (Lemma \ref{lem:ram}): 
if $F_1$ and $F_2$ are finite extensions of a complete discretely valued field $F$ 
and the residue field extension of $F_1\mid F$ is separable, then $e(F_1F_2\mid F_2)\leq e(F_1\mid F)$. 
This inequality is folklore and we claim no novelty for it; what matters here, and what we could not find recorded, 
is that it needs no tameness and  separability of one of the two residue field extensions only. Both features are sharp: 
for a tame $F_1\mid F$ one has the stronger divisibility of Abhyankar's lemma, which fails in the wild case (Example \ref{rem:wild}), 
and the separability hypothesis cannot be dropped (Example \ref{rem:sharp}). 
We say more about the provenance of the lemma in Section \ref{sec:ram}.
This is the only point at which the perfectness of the residue field of
$V$ enters the proof of Theorem \ref{thm:main}, perfectness being what guarantees
the separability hypothesis for all the extensions we meet; it is used again, on
its own account, in Proposition \ref{prop:nonvacuous}.
The lemma itself enters the argument at a single step: attached to a residually
algebraic torsion extension $W$ there is a sequence $\{(s_n,\delta_n)\}$ of minimal pairs
defining it, and the whole question is whether the sequence $\{s_n\}_{n\in\N}$ converges in $\bK$; the lemma bounds
the denominators of the $\delta_n$ when $W$ is discrete, and convergence follows.

Two further examples delimit the main result. First, not every $V_{\alpha}$,
for $\alpha\in\bK$, is a DVR: over $\Q_p$ we exhibit an $\alpha$ in the completion $\C_p$
of $\overline{\Q_p}$ for which the value group of $V_{\alpha}$ is non-discrete (Example \ref{ex:nondiscrete}),
so that the extensions described by Theorem \ref{thm:main} are exactly the valuation domains
$V_{\alpha}$ which happen to be a DVR. Secondly, and more to the point, the hypothesis
of discreteness cannot be omitted from Theorem \ref{thm:main}: if $V=\Z_p$, then
there is a residually algebraic torsion extension of $V$ to $K(X)$ which is not of the
form $V_{\alpha}$ for any $\alpha$ (Example \ref{prop:notValpha}). The
construction uses the fact that $\mathbb{C}_p$ is not spherically complete; as we
explain in Remark \ref{rem:IZ}, it also shows that the identification of the
completions of $K(X)$ given in \cite{IZ95} for all residually algebraic torsion
extensions requires a hypothesis which is not stated there, and that discreteness
is one such hypothesis.

Theorem \ref{thm:main} thus leaves open which $V_{\alpha}$ are DVRs. We answer
this in Proposition \ref{prop:criterion}: for
$\alpha\in\bK$ transcendental over $K$, the valuation domain $V_{\alpha}$ is
a DVR if and only if $\alpha$ is the limit of a Cauchy sequence $\{s_n\}_{n\in\N}\subset\oK$
with bounded set of ramification indices $\{e(\K(s_n)\mid\K)\}_{n\in\N}$.
Together with Theorem \ref{thm:main} this describes that class completely,
and we record the description as Corollary  \ref{cor:description},
the form in which the results of this paper are perhaps best read.
Finally, and most interestingly, we show in Theorem \ref{bounded ramification subfield}
that the set of $\alpha\in\bK$ for which $V_{\alpha}$ is a DVR
is a (proper) subfield of $\bK$, properly containing $\oK$. 
For the two families of \S\ref{subsec:families} this subfield can be named outright:
in equal characteristic zero it is the field of Puiseux series over $\ok$
(Example \ref{ex:puiseux}).

We close this introduction by saying plainly what is new in this paper and what is not. The case of a finite
residue field extension in Theorem \ref{thm:main} is \cite[Theorem 2.5]{PerTransc}, and the
compositum bound of Lemma \ref{lem:ram} seems to be folklore;
the contributions of the paper are the following five.
(i) \emph{Theorem \ref{thm:main} for an arbitrary DVR with perfect residue
field.} For $V=\Z_{(p)}$ this is \cite[Corollary 2.28]{p25}, and the passage to a general
base is not a formality: in equal characteristic $p$ the extension $\oK\mid\K$ need not be
separable, so that $G=\Aut(\oK\mid\K)$ is not a Galois group and the argument must be carried
by the two properties (G1) and (G2) of Section \ref{sec:notation} rather than by Galois
theory; the second bullet of Example \ref{ex:nonvacuous} exhibits exactly this situation.
This is the main point of the paper. (ii) \emph{Proposition \ref{prop:criterion}}, which
decides, for a given $\alpha$, whether $V_{\alpha}$ is a DVR, and so answers the question
that Theorem \ref{thm:main} leaves open; combined with the theorem it gives the complete
description of Corollary \ref{cor:description}. (iii)
\emph{Theorem \ref{bounded ramification subfield}}, that $\bK^{br}$
is a subfield of $\bK$, together with Theorem \ref{thm:union} and Example \ref{ex:puiseux}.  
(iv) \emph{Example \ref{prop:notValpha}}, which shows that the hypothesis of discreteness cannot simply be
dropped from Theorem \ref{thm:main}, and which is what allows us to locate, in Remark
\ref{rem:IZ}, a hypothesis missing from \cite{IZ95}. (v) \emph{Examples \ref{rem:wild} and
\ref{rem:sharp}}, which delimit Lemma \ref{lem:ram} sharply: its conclusion cannot be
strengthened to the divisibility of the tame Abhyankar lemma, and its separability hypothesis
cannot be removed.

\section{Notation and preliminaries}\label{sec:notation}

From now on $V$ is a DVR, $K$ its quotient field, $v$ the associated valuation,
$\pi$ a fixed generator of the maximal ideal of $V$ and $k_v=V/\pi V$ its residue
field; $X$ is an indeterminate over $K$. The letter $p$ is reserved throughout for a
rational prime and for the residue characteristic, and is never used for the uniformizer;
note that $\pi$ is a uniformizer of the completion $\widehat V$ as well as of $V$. By a DVR we mean a
valuation domain which is discrete of rank one, so that $\Gamma_v\cong\Z$; we use the two
descriptions interchangeably, for the extensions of $V$ considered below as well as for $V$
itself. The distinction is worth making, since Section \ref{sec:monomial} exhibits a discrete
valuation domain of rank two. We fix, once and for
all, the following
notation, every symbol keeping the meaning given here until the end of the
paper.

\begin{itemize}
\item $\Gamma_v$ is the value group of $v$ and
      $\Gamma_{\overline v}=\Gamma_v\otimes_{\Z}\Q\cong\Q$ its divisible hull.
\item $\K$ and $\widehat V$ are the $v$-adic completions of $K$ and of $V$;
      $\K$ is henselian (see for example \cite[Theorem 1, p. 98]{Rib}) and has the same residue field $k_v$ as $V$.
\item $\oK$ is an algebraic closure of $\K$ and $\oV$ is the integral closure of
      $\widehat V$ in $\oK$, a rank one non-discrete valuation domain. We denote
      again by $v$ the \emph{unique} extension of $v$ to $\oK$, and by $\ok$ the
      residue field of $\oV$, an algebraic closure of $k_v$.
\item $\bK$ and $\bO$ are the completions of $\oK$ and of $\oV$, respectively. Since a valued
      field of rank one and its completion have the same value group and the
      same residue field, $v(\bK^{\times})=\Gamma_{\overline v}$ and the residue
      field of $\bO$ is $\ok$. When $V=\Z_{(p)}$ one has $\bK=\mathbb{C}_p$.
\item $G=\Aut(\oK\mid\K)$; see the comment below in \S \ref{sec:group G}.
\item $e(F_1|F_2)$ is the ramification index of a valued field extension $F_1|F_2$,
namely, the index of the corresponding value groups,
where $F_1|F_2$ means the field extension of $F_2 \subseteq F_1$.
\end{itemize}

We work over $\oK$ rather than over an algebraic closure of $K$ because it is
only over $\oK$ that the extension of $v$ is unique, so that $v\circ\sigma=v$
for every $\sigma\in G$. Note also that every element of $\bK\setminus\oK$ is
transcendental over $\oK$.

\subsection{Roots of unity whose order is prime to the residue characteristic} \label{subsec:roots}
Assume $\operatorname{char}k_v=p>0$, let $m\geq1$ be prime to $p$ and let $\mu_m$ denote the
group of $m$-th roots of unity of $\oK$. Every $\zeta\in\mu_m$ is a unit, since
$m\,v(\zeta)=0$ forces $v(\zeta)=0$. Moreover the residue map is injective on $\mu_m$: if
$\zeta\in\mu_m$ has residue $1$ and $\zeta\neq1$, then $1+\zeta+\cdots+\zeta^{m-1}=0$,
whereas its residue is $m$, which is non-zero in $k_v$ because $p\nmid m$. Equivalently,
$$v(1-\zeta)=0\qquad\text{for every }\zeta\in\mu_m,\ \zeta\neq1 ,$$
and, the residue map being an injective homomorphism on $\mu_m$, the residue of a primitive
$m$-th root of unity is again a primitive $m$-th root of unity. We use this in Example
\ref{ex:nonvacuous} and in Example \ref{ex:nondiscrete}.

\subsection{The ramification index of an element of $\bK$}\label{sec:eps}

For $s\in\bK$, we put
$$e_s=e(\K(s)\mid \K),$$
which in general is an element of $\N\cup\{\infty\}$ and we call it the \emph{ramification index of $s$ over $\K$}. 
Clearly, $e_s\in\N$ whenever $s\in \oK$ because in that case the extension $\K(s)\mid\K$ is finite.
In general, the extension $\K(s)\mid \K$ is transcendental precisely when $s\not\in\oK$.
We will show that even in this case, $e_s$ might be finite, and we will characterize in \S \ref{sec:examples}
precisely when this is the case (see  Proposition \ref{prop:criterion}).
In the following we will make use of the following fact: $e_s$ is finite precisely
when the valuation domain $\bO\cap\K(s)$ of $\K(s)$ is a DVR.

\subsection{On the group $G$}\label{sec:group G}
Let $F\subseteq L$ be an extension of fields and let $\Aut(L\mid F)$ denote the
group of all automorphisms $\sigma$ of $L$ such that $\sigma(a)=a$ for all
$a\in F$. We say that $L$ is \emph{quasi-Galois} over $F$ if it is a normal (or
splitting) extension of $F$. The extension $\oK\mid\K$ is quasi-Galois, but in
general it is not Galois, so that $G$ is not a Galois group: if
$\operatorname{char}\K=p>0$ and $\K$ is not perfect, then $\oK\mid\K$ is not
separable, the restriction to $\K^{\mathrm{sep}}$, the separable closure of $\K$
in $\oK$, identifies $G$ with
$\Gal(\K^{\mathrm{sep}}\mid\K)$, and the fixed field of $G$ is the perfect
closure of $\K$ and not $\K$. This is immaterial for what follows, since only
the following two properties are used, and both of them hold for quasi-Galois
extensions.

\begin{itemize}
\item[(G1)] $G$ acts transitively on the set of the roots in $\oK$ of an
irreducible polynomial of $\K[X]$ and, more generally, the automorphism group of
a quasi-Galois algebraic extension acts transitively on the set of the
extensions of a valuation of the base field to it. We shall use the latter
property for the extensions of a valuation of $\K(X)$ to $\oK(X)$; note that
$\Aut(\oK(X)\mid\K(X))=G$, because $\oK$ is the algebraic closure of $\K$ in
$\oK(X)$, hence is stable under every automorphism fixing $\K(X)$, while
conversely every $\sigma\in G$ extends to $\oK(X)$ coefficientwise with
$\sigma(X)=X$.
\item[(G2)] $v\circ\sigma=v$ for every $\sigma\in G$, by the uniqueness of the
extension of $v$ to $\oK$. Consequently every $\sigma\in G$ is an isometry of
$\oK$ and extends by continuity to $\bK$.
\end{itemize}

\subsection{Monomial valuations}\label{sec:monomial}
Let $U$ be a rank one valuation domain with quotient field $F$ and valuation
$u$, let $\alpha\in F$ and let $\delta$ be an element of the divisible hull of
the value group of $u$. For $g\in F[X]$,
written as $$g(X)=a_n(X-\alpha)^n+\cdots+a_1(X-\alpha)+a_0$$ in $F[X]$, set
$$u_{\alpha,\delta}(g)=\min\{u(a_i)+i\delta\mid i=0,\ldots,n\}.$$
The function $u_{\alpha,\delta}$ extends to a valuation of $F(X)$
\cite[Ch.~VI, \S10, no.~1, Lemma~1]{Bourb}, called a \emph{monomial valuation}, of
rank one, and we denote by $U_{\alpha,\delta}$ the associated valuation domain;
clearly $U_{\alpha,\delta}\cap F=U$. We shall use this construction for $U=\oV$,
which is a rank one valuation domain as well.

The limiting case $\delta=\infty$ is of a different nature, and we treat it
separately. The formula above then reads $u_{\alpha,\infty}(g)=u(g(\alpha))$,
which takes the value $\infty$ at the non-zero polynomial $X-\alpha$ and is
therefore \emph{not} a valuation of $F(X)$. The set
$$U_{\alpha,\infty}=\{\phi\in F(X)\mid \phi \text{ is defined at }\alpha
\text{ and } \phi(\alpha)\in U\}$$
is nevertheless a valuation domain of $F(X)$ lying over $U$: it is the pullback
of $U$ under the residue map of the $(X-\alpha)$-adic valuation domain
$F[X]_{(X-\alpha)}$ of $F(X)$, whose residue field is $F$. Being the composite of
two valuation domains of rank one, $U_{\alpha,\infty}$ has rank two, and its
residue field is that of $U$ itself, so that it is trivially residually algebraic
over $U$; the caveat in the display is needed because
$\alpha$ lies in $F$, so that $X-\alpha$ is a prime element of $F[X]$ and a
rational function may have a pole at $\alpha$. By contrast, in Section
\ref{sec:main} we shall attach a
\emph{rank one} valuation domain $V_{\alpha}$ of $K(X)$ to elements $\alpha$ of
$\bK$ that are transcendental over $K$, and the two constructions should not be
confused: for $\alpha\in\oK$ transcendental over $K$ the relation between them is
$V_{\alpha}=\oV_{\alpha,\infty}\cap K(X)$, as we shall see.

An extension $W$ of $V$ to $K(X)$ is \emph{torsion} if
$\Gamma_w\subseteq\Gamma_{\overline v}$, where $\Gamma_w$ is the value group of
$W$; equivalently, if $\Gamma_w/\Gamma_v$ is a torsion group. Every residually
transcendental extension of $V$ to $K(X)$ is torsion, and not every torsion extension is
residually transcendental.
Indeed, by Abhyankar's inequality
\cite[Ch.~VI, \S10, no.~3, Cor.~1]{Bourb},
$\ratrk(\Gamma_w/\Gamma_v)+\operatorname{trdeg}(k_w\mid k_v)\leq
\operatorname{trdeg}(K(X)\mid K)=1$, so that a residually transcendental
extension has $\ratrk(\Gamma_w/\Gamma_v)=0$, that is, it is torsion; and the second
assertion follows from Remark \ref{rem:Valpha},
applied to any $\alpha\in\bK\setminus\oK$. Such an $\alpha$ exists, and is
transcendental over $K$: an algebraically closed field carrying a valuation of rank one whose
value group is not discrete is never complete, so that $\bK\setminus\oK\neq\emptyset$, and
every element of $\bK$ algebraic over $K$ lies in $\oK$.
If a residually transcendental extension $W$ of $V$ to $K(X)$
contains $V[X]$, then $W=\oV_{\alpha,\delta}\cap K(X)$ for
some $(\alpha,\delta)\in\oK\times\Gamma_{\overline{v}}$
\cite[Theorem 1.6 and Corollary 2.11]{ChPer}.
\section{A bound for the ramification index in a compositum}\label{sec:ram}

The notation of this section is local to it and is unrelated to the one fixed in
Section \ref{sec:notation}. Throughout, $(F,\nu)$ is a valued field which is
complete
with respect to the discrete valuation $\nu$ of rank one, $\kappa$ is its residue
field and $\overline F$ is an algebraic closure of $F$. Since $F$ is complete,
$\nu$ admits a unique extension to $\overline F$, and it is with respect to this
extension that all the ramification indices below are taken. We use freely two
standard properties of such a field (see, for example, \cite{ZS2}): it is henselian, and
every finite extension of it is again complete with respect to a discrete
valuation of rank one; and $[E:F]=e(E\mid F)\cdot f(E\mid F)$,
where $f(E\mid F)$ is the residue degree of the valued field extension $E|F$, for every finite
extension $E$ of $F$, no separability being required.

\begin{lemma}\label{lem:ram}
Let $F_1,F_2$ be finite extensions of $F$, let $L=F_1F_2$ be their compositum in
$\overline F$, and put $e_1=e(F_1\mid F)$ and $e=e(L\mid F_2)$. If the residue
field extension of $F_1\mid F$ is separable, then
$$e\leq e_1 .$$
In particular the inequality holds for all $F_1$ and $F_2$ as soon as $\kappa$ is
perfect.
\end{lemma}

\begin{proof}
Write $f_1=f(F_1\mid F)$ and, for a subfield $E$ of $\overline F$ containing $F$,
$O_E$ and $\kappa_E$ for the valuation ring of $E$ and its residue field.

\vspace{.2cm}
\noindent
{\bf Step 1.} {\em $F_1$ contains an unramified extension of $F$ of degree
$f_1$.} Let $\ell$ be the residue field of $F_1$. By hypothesis $\ell\mid\kappa$
is separable, and it is finite, of degree $f_1$, because $[F_1:F]<\infty$; by the
primitive element theorem $\ell=\kappa(\overline\theta)$ for some
$\overline\theta$. Let $g\in O_F[X]$ be a monic lift of the minimal polynomial of
$\overline\theta$ over $\kappa$, of degree $f_1$. Then $\overline g$ is separable,
so that $\overline\theta$ is a simple root of it, and Hensel's lemma in the henselian
field $F_1$ yields a root $\theta\in O_{F_1}$ of $g$ whose residue is $\overline\theta$. Put
$F_1^{\mathrm{ur}}=F(\theta)$; since $\theta\in O_{F_1}$, this is a subfield of
$F_1$, a fact used in Step 3. Then $[F_1^{\mathrm{ur}}:F]\leq\deg g=f_1$, while
the residue field of $F_1^{\mathrm{ur}}$ contains $\kappa(\overline\theta)=\ell$,
so that $f(F_1^{\mathrm{ur}}\mid F)\geq f_1$; by the degree formula,
$[F_1^{\mathrm{ur}}:F]=f_1$, $e(F_1^{\mathrm{ur}}\mid F)=1$ and the residue field
of $F_1^{\mathrm{ur}}$ is $\ell$. Again by the degree formula,
$$[F_1:F_1^{\mathrm{ur}}]=\frac{[F_1:F]}{[F_1^{\mathrm{ur}}:F]}
=\frac{e_1f_1}{f_1}=e_1 .$$

\vspace{.2cm}
\noindent
{\bf Step 2.} {\em $F_2':=F_1^{\mathrm{ur}}F_2=F_2(\theta)$ is unramified
over $F_2$.} Let $h\in O_{F_2}[X]$ be the minimal polynomial of $\theta$ over
$F_2$; it is monic with coefficients in $O_{F_2}$, because $\theta$ is integral
over $O_F$ and $O_{F_2}$ is integrally closed. Moreover $h$ divides $g$ in
$O_{F_2}[X]$: writing $g=hq$ with $q\in F_2[X]$ monic, Gauss's lemma over the
valuation ring $O_{F_2}$ gives $q\in O_{F_2}[X]$, both $g$ and $h$ being monic
with coefficients in $O_{F_2}$. Hence $\overline h$ divides $\overline g$ in
$\kappa_{F_2}[X]$; the latter is separable, being the reduction of the separable
$\overline g\in\kappa[X]$ and separability being preserved under extension of the
base field, so $\overline h$ is separable too; a factorization of
$\overline h$ into two non-constant factors would be a factorization into coprime
factors, which Hensel's lemma would lift to a factorization of $h$, against its
irreducibility. So $\overline h$ is irreducible, the residue field of $F_2'$
contains $\kappa_{F_2}(\overline\theta)$, which has degree
$\deg\overline h=\deg h=[F_2':F_2]$ over $\kappa_{F_2}$, and consequently
$f(F_2'\mid F_2)=[F_2':F_2]$ and $e(F_2'\mid F_2)=1$.

\vspace{.2cm}
\noindent
{\bf Step 3.} {\em Conclusion.} As $F_1^{\mathrm{ur}}\subseteq F_1$ we have
$L=F_1F_2=F_1F_1^{\mathrm{ur}}F_2=F_1F_2'$, and $F_1^{\mathrm{ur}}\subseteq
F_2'$; writing $F_1=\sum\limits_{i=1}^{e_1}F_1^{\mathrm{ur}}\omega_i$, the $F_2'$-span
$A=\sum\limits_iF_2'\omega_i$ is closed under multiplication, because
$\omega_i\omega_j\in F_1=\sum\limits_kF_1^{\mathrm{ur}}\omega_k$ and
$F_1^{\mathrm{ur}}\subseteq F_2'$. Thus $A$ is a finite-dimensional $F_2'$-algebra
which is an integral domain, being contained in $L$, hence a field; as it contains $F_2'$
and $F_1$ we get $A=F_1F_2'=L$, so that $[L:F_2']\leq e_1$. Therefore
$$e=e(L\mid F_2)=e(L\mid F_2')\cdot e(F_2'\mid F_2)=e(L\mid F_2')\leq[L:F_2']
\leq e_1,$$
which is the assertion.
\end{proof}

When $F_1\mid F$ is tame, that is, when in addition $\operatorname{char}\kappa\nmid e_1$, 
the inequality is a form of Abhyankar's lemma and one even gets the divisibility $e\mid e_1$, 
together with $e(F_1F_2\mid F)=\operatorname{lcm}\big(e(F_1\mid F),e(F_2\mid F)\big)$: see \cite{ChHa} 
for local fields --- a setting narrower than the complete discretely valued fields of this section, 
so that the comparison offered there is with a special case --- and \cite[Theorem 2]{DK20} 
for valued fields whose value group has rational rank one. Tameness is not available in Section \ref{sec:main}, 
where the extensions occurring might be wildly ramified in general.
Steps 1 and 2 above are nothing but the standard theory of unramified extensions 
of a complete discretely valued field --- the
maximal unramified subextension of $F_1\mid F$ exists and has degree $f_1$, and
unramifiedness is preserved by base change to $F_2$ --- for which we refer to
\cite[Ch.~II, \S7]{Neu} or \cite[Ch.~III]{Ser}, and the same two steps are carried
out in \cite[Lemmas 6 and 7]{SK20}. Granting them, the lemma follows in two lines:
$F_2'=F_1^{\mathrm{ur}}F_2$ is unramified over $F_2$, and $F_1^{\mathrm{ur}}
\subseteq F_1\cap F_2'$ with $L=F_1F_2'$, so that
$$e(L\mid F_2)=e(L\mid F_2')\leq[L:F_2']\leq[F_1:F_1^{\mathrm{ur}}]=e_1,$$
the last inequality being the elementary bound $[AB:B]\leq[A:A\cap B]$ for the
degree of a compositum. We have written the argument out in full only because we
could not locate the statement in the generality we need --- no tameness, and
separability of one of the two residue field extensions only; \cite[Theorems 2 and
3]{SK20}, for instance, compare $e(F_1F_2\mid F)$ with $e(F_1\mid F)e(F_2\mid F)$
under a linear disjointness hypothesis that we cannot impose here.

Neither the conclusion of Lemma \ref{lem:ram} nor its hypothesis can be improved,
and the two examples that follow show why. In Example \ref{rem:wild} the extension
$F_1\mid F$ is wildly ramified and the divisibility $e\mid e_1$ fails, so that the inequality
is the right statement; in Example \ref{rem:sharp} the separability hypothesis, which is used
only in Step 1, is dropped and the bound itself fails.

\begin{example}\label{rem:wild}
Let $p$ be an odd prime and $F=\Q_p$; let $p^{1/p}$ be a root of the Eisenstein
polynomial $X^p-p$, let $\zeta$ be a primitive $p$-th root of unity, and put
$$F_1=F(p^{1/p}),\qquad F_2=F(\zeta p^{1/p}),\qquad L=F_1F_2 .$$
Then $F_1\mid F$ and $F_2\mid F$ are totally ramified of degree $p$, so that
$e_1=e(F_2\mid F)=p$ and both residue field extensions are trivial, hence separable; and
$$e=e(L\mid F_2)=p-1 ,$$
which does not divide $e_1$. Moreover $e(L\mid F)=p(p-1)$ is not the least common multiple of
$e_1$ and $e(F_2\mid F)$.
\end{example}

\begin{proof}
The polynomial $X^p-p$ is Eisenstein, so $F_1\mid F$ is totally ramified of degree
$p$, and the same holds for $F_2$, whose generator $\zeta p^{1/p}$ is again a root of
$X^p-p$. The compositum
$L=F_1F_2=F(p^{1/p},\zeta)$ has degree $p(p-1)$ over $F$, the degrees $p$ and
$p-1$ being coprime, and it is again totally ramified, since it contains two
totally ramified subextensions of coprime degrees. Hence
$$e=e(L\mid F_2)=\frac{p(p-1)}{p}=p-1<p=e_1 ,$$
and $p-1$ does not divide $p$; and $e(L\mid F)=p(p-1)$, whereas the least common
multiple of $e_1$ and $e(F_2\mid F)$ is $p$.
\end{proof}

The assertion that $p-1$ does not divide $e_1$ is the only point at which the
hypothesis that $p$ be odd is used, and it is indispensable: for $p=2$ one has $\zeta=-1$, so
that $F_2=F_1$, $L=F_1$ and $e=1$, which does divide $e_1=2$. 
That $e(L\mid F)$ is not the least common multiple of the two
ramification indices
is consistent with \cite[Theorem 2]{DK20}, whose hypotheses exclude the present
example, the extension $F_1\mid F$ being wildly ramified.

\begin{example}\label{rem:sharp}
Let $p$ be a prime, let $u$ be an indeterminate over $\mathbb{F}_p$ and put
$\kappa_0=\mathbb{F}_p(u)$ and $F=\kappa_0((t))$ with the $t$-adic valuation, $\nu(t)=1$;
the field $F$ is complete, $\nu$ is discrete of rank one, and the residue field
$\kappa=\kappa_0$ is imperfect. Put
$$F_1=F(x)\ \text{ with }\ x^p=u,\qquad F_2=F(y)\ \text{ with }\ y^p=u+t,
\qquad L=F_1F_2 .$$
Then $[F_1:F]=[F_2:F]=p$, both residue field extensions are purely inseparable of degree
$p$, and $e_1=e(F_1\mid F)=1$, whereas
$$e=e(L\mid F_2)\geq p>1=e_1 .$$
\end{example}

\begin{proof}
The residue field $\kappa_0$ is imperfect because $\kappa_0^p=\mathbb{F}_p(u^p)$
does not contain $u$.
Neither $u$ nor $u+t$ lies in $F^p=\mathbb{F}_p(u^p)((t^p))$: the first because
$u\notin\kappa_0^p$, the second because every exponent occurring in an element of $F^p$ is
divisible by $p$. Hence $[F_1:F]=[F_2:F]=p$; the residues of
$x$ and of $y$ are both the $p$-th root of $u$ in $\overline{\kappa}$, so that
$f(F_i\mid F)=p$ and hence, by the degree formula, $e(F_i\mid F)=1$ for
$i=1,2$. In particular $e_1=e(F_1\mid F)=1$ and $\Gamma_{F_2}=\Gamma_F$. On the
other hand, in characteristic $p$,
$$(y-x)^p=y^p-x^p=t ,$$
so that $L$ contains a $p$-th root of $t$ and therefore
$e(L\mid F_2)\geq p$.
\end{proof}

Thus the conclusion of Lemma \ref{lem:ram} fails, already for extensions of
degree $p$, as soon as the residue field extension of $F_1\mid F$ is allowed to
be inseparable: here it is purely inseparable. This is also why
\cite[Theorem 2]{DK20} does not apply to the present example any more than it
does to the previous one, although the ramification index of $F_1\mid F$ is $1$:
tameness requires the residue field extension to be separable, and not merely
that $\operatorname{char}\kappa\nmid e_1$.

\section{Discrete residually algebraic extensions}\label{sec:main}

The main result is Theorem \ref{thm:main}: for $k_v$ perfect, every DVR of
$K(X)$ extending $V$ and residually algebraic over it has the form
$V_{\alpha}$ for an
$\alpha\in\bK$ transcendental over $K$, with $\alpha\in\oK$ exactly when the
residue field extension is finite.  We construct $V_{\alpha}$ in
\S\ref{subsec:Valpha}, prove the theorem in \S\ref{subsec:mainthm}, settle in
\S\ref{subsec:cases} which case occurs, and illustrate both in
\S\ref{subsec:families}.

\subsection{The valuation domains {$V_{\alpha}$}}\label{subsec:Valpha}

For an element $\alpha\in\bK$ which is transcendental over $K$ we set
$$V_{\alpha}=\{\phi\in K(X)\mid \phi(\alpha)\in\bO\}
=\{\phi\in K(X)\mid v(\phi(\alpha))\geq0\},$$
which is well defined, since $g(\alpha)\neq0$ for every non-zero $g\in K[X]$, and
which is a valuation domain of $K(X)$ lying over $V$, being the pullback of
$\bO\cap K(\alpha)$ under the isomorphism $K(X)\to K(\alpha)$,
$\phi\mapsto\phi(\alpha)$. In particular $V_{\alpha}$ has rank one and its value
group is $v(K(\alpha)^{\times})$. For $\alpha\in\oK$ transcendental over $K$ one
has
\begin{equation}\label{eq:trace}
V_{\alpha}=\oV_{\alpha,\infty}\cap K(X),
\end{equation}
in the notation of Section \ref{sec:monomial}: both sides consist of the
$\phi\in K(X)$ with $v(\phi(\alpha))\geq0$, and no rational function of $K(X)$ has
a pole at $\alpha$, $\alpha$ being transcendental over $K$. We stress that the
rank two valuation domain $\oV_{\alpha,\infty}$ of $\oK(X)$ and its rank one
contraction $V_{\alpha}=\oV_{\alpha,\infty}\cap K(X)$ are different objects.

\begin{remark}\label{rem:conj}
Let $\alpha\in\bK\setminus\oK$. Then $\alpha$ is transcendental over $\oK$, the
field $\oK$ being algebraically closed, so that
$$\overline{V_{\alpha}}=\{\phi\in\oK(X)\mid v(\phi(\alpha))\geq0\}$$
is defined; it is the pullback of $\bO\cap\oK(\alpha)$ under the isomorphism
$\oK(X)\to\oK(\alpha)$, hence a valuation domain of $\oK(X)$ of rank one lying
over $\oV$, and its restriction to $K(X)$ is $V_{\alpha}$. Let $\sigma\in G$,
extended by continuity to $\bK$ as in (G2); then $\sigma$ maps $\oK$ onto itself,
so $\sigma(\alpha)\in\bK\setminus\oK$, and
\begin{equation}\label{eq:conj}
\sigma(\overline{V_{\alpha}})=\overline{V_{\sigma(\alpha)}} .
\end{equation}
Indeed, for $\phi\in\oK(X)$ one has $(\sigma^{-1}\phi)(\alpha)=
\sigma^{-1}\big(\phi(\sigma(\alpha))\big)$, whence, $v$ being $\sigma$-invariant,
$$\phi\in\sigma(\overline{V_{\alpha}})\iff
v\big((\sigma^{-1}\phi)(\alpha)\big)\geq0\iff
v\big(\phi(\sigma(\alpha))\big)\geq0\iff\phi\in\overline{V_{\sigma(\alpha)}} .$$
\end{remark}

\begin{remark}\label{rem:Valpha}
Let $\alpha\in\bK$ be transcendental over $K$. Then $V_{\alpha}$ is a residually algebraic torsion
extension of $V$ to $K(X)$.
Indeed, the value group of
$V_{\alpha}$ is $v(K(\alpha)^{\times})\subseteq v(\bK^{\times})
=\Gamma_{\overline v}$ and its residue field embeds into the residue field
$\ok$ of $\bO$, which is algebraic over $k_v$. The same inclusion gives the rank: a
non-trivial subgroup of $\Gamma_{\overline v}\cong\Q$ has no non-trivial convex subgroup.
This argument uses nothing special about $V_{\alpha}$: being torsion
over $V$ means precisely that $\Gamma_w$ lies in $\Gamma_{\overline v}$, so that every
torsion extension of $V$ to $K(X)$ has rank one.
\end{remark}

The next lemma identifies the completion of $K(X)$ 
with respect to $V_{\gamma}$ for $\gamma\in\bK$, transcendental over $K$.
It is what makes assertion (2) of Theorem \ref{thm:main} independent of the choice of
$\alpha$, through Corollary \ref{lem:rank}, and it will be used again in Remark
\ref{rem:uniqueness} and in Proposition \ref{rem:byproduct}. Throughout, ``closure'' means closure
in the $v$-adic topology, not algebraic closure.

\begin{lemma}\label{lem:completion}
Let $\gamma\in\bK$ be transcendental over $K$ and let $C_{\gamma}$ be the closure of
$K(\gamma)$ in $\bK$. Then the evaluation $\phi\mapsto\phi(\gamma)$ identifies the completion
of $K(X)$, with respect to the valuation associated with $V_{\gamma}$, with $C_{\gamma}$;
moreover $\K\subseteq C_{\gamma}$, and
$$C_{\gamma}=\K(\gamma)\ \text{ if }\ \gamma\in\oK,\qquad
[C_{\gamma}:\K]=\infty\ \text{ if }\ \gamma\notin\oK .$$
In particular $[C_{\gamma}:\K]<\infty$ precisely when $\gamma\in\oK$.
\end{lemma}

\begin{proof}
The evaluation $\phi\mapsto\phi(\gamma)$ is an isomorphism
of the valued field $K(X)$, equipped with the valuation attached to $V_{\gamma}$,
onto the subfield $K(\gamma)$ of $\bK$ with the valuation induced by $v$; as a
closed subfield of the complete field $\bK$, the field $C_{\gamma}$ is complete,
and it therefore realizes the completion of $K(X)$ with respect to that
valuation. Note that $K\subseteq\K\subseteq\oK\subseteq\bK$ and that $\K$, being
complete, is closed in $\bK$: the closure of $K$ in $\bK$ is thus $\K$ itself, and
$\K\subseteq C_{\gamma}$.

Suppose first that $\gamma\in\oK$. Then $\K(\gamma)$ is a finite extension of the
complete field $\K$,
hence is itself complete and so closed in $\bK$; as it contains $K(\gamma)$ we get
$C_{\gamma}\subseteq\K(\gamma)$, while $C_{\gamma}$ is a field containing $\K$ and
$\gamma$, so that $C_{\gamma}=\K(\gamma)$ and $[C_{\gamma}:\K]<\infty$. Suppose now
that $\gamma\notin\oK$. Then $\gamma\in C_{\gamma}$, so $[C_{\gamma}:\K]<\infty$ would make
$\gamma$
algebraic over $\K$; but every element of $\bK$ algebraic over $\K$ lies in $\oK$,
the field $\oK$ being algebraically closed. Hence
$[C_{\gamma}:\K]=\infty$.
\end{proof}

\begin{corollary}\label{lem:rank}
Let $\alpha\in\oK$ and $\beta\in\bK\setminus\oK$ be transcendental over $K$.
Then $V_{\alpha}\neq V_{\beta}$.
\end{corollary}

\begin{proof}
Suppose $V_{\alpha}=V_{\beta}$. Composing the two identifications
provided by Lemma \ref{lem:completion} we obtain an isomorphism
$\iota\colon C_{\alpha}\to
C_{\beta}$ of valued fields which fixes $K$ pointwise; being an isometry, $\iota$
is continuous, hence it fixes pointwise the closure $\K$ of $K$. Therefore
$\iota$ is an isomorphism of extensions of $\K$, and
$$[C_{\alpha}:\K]=[C_{\beta}:\K].$$
By Lemma \ref{lem:completion} the left-hand side is finite and the right-hand side
is infinite, a contradiction.
\end{proof}

\begin{remark}\label{rem:notalgebraic}
One cannot argue instead by comparing the extensions of $V_{\alpha}$ to $\oK(X)$,
although $\oV_{\alpha,\infty}$ has rank two by \eqref{eq:trace} and
Section \ref{sec:monomial} whereas $\overline{V_{\beta}}$ has rank one by
Remark \ref{rem:conj}. Such a
comparison would require the two extensions to be conjugate, and the transitivity
statement (G1) of Section \ref{sec:notation} is available only over $\K(X)$: the
extension $\oK(X)\mid K(X)$ is \emph{not} algebraic unless $K=\K$, since $\K\mid
K$ is transcendental whenever $K$ is not complete.
\end{remark}

\subsection{\texorpdfstring{The main theorem}{The main theorem}}\label{subsec:mainthm}

We are now ready to prove the main result of this paper, which generalizes
\cite[Corollary 2.28]{p25} and \cite[Theorem 2.5]{PerTransc}.

\begin{theorem}\label{thm:main}
Let $V$ be a DVR with quotient field $K$ and perfect residue field, and let $W$
be a DVR with quotient field $K(X)$ which
extends $V$ and is residually algebraic over $V$. Denote by $k_w$ and $k_v$ the
residue fields of $W$ and of $V$. Then:
\begin{enumerate}
\item there exists $\alpha\in\bK$, transcendental over $K$, such that $W=V_{\alpha}$;
\item for every such $\alpha$, one has $\alpha\in\oK\iff[k_w:k_v]<\infty$;
\item when $[k_w:k_v]=\infty$, every such $\alpha$ is the limit of
a Cauchy sequence $\{s_n\}\subset\oK$ with $e_{s_n}=[\Gamma_w:\Gamma_v]$ for all large $n$.
\end{enumerate}
\end{theorem}

\begin{proof}
Since $V$ is a DVR we have $\Gamma_v\cong\Z$; as $W$ extends $V$, the group
$\Gamma_v$ is a non-trivial subgroup of $\Gamma_w\cong\Z$, so that
$e:=[\Gamma_w:\Gamma_v]<\infty$.

Assume first that $[k_w:k_v]<\infty$. We recall the short argument of
\cite[Theorem 2.5]{PerTransc}, which in this case is quicker than the appeal to that theorem
and needs no separability. Put $f=[k_w:k_v]$, let $L=\widehat{K(X)}$ be the completion of
$K(X)$ with respect to $w$ and identify $\K$ with the closure of $K$ in $L$. Completions are
immediate extensions, so that $e(L\mid\K)=e$ and $f(L\mid\K)=f$, both finite; for complete
discretely valued fields this forces $[L:\K]=ef<\infty$. Indeed, if
$\omega_1,\ldots,\omega_f\in O_L$ lift a basis of $k_w$ over $k_v$ and $\varpi$ is a
uniformizer of $L$, successive approximation writes every element of $O_L$ as a convergent
$O_{\K}$-combination of the $\omega_i\varpi^{j}$, $0\leq j<e$. Hence $X\in L$ is algebraic
over $\K$; any $\K$-embedding $\iota\colon L\to\oK$ is isometric, the extension of $v$ to
$\oK$ being unique, so that $\alpha:=\iota(X)$ satisfies $w(\phi)=v(\phi(\alpha))$ for every
$\phi\in K(X)$, that is, $W=V_{\alpha}$. Moreover $\alpha\in\oK$ is transcendental over $K$,
being the image of $X$ under a field embedding fixing $K$, and the case $\alpha=\infty$ of
\cite[Theorem 2.5]{PerTransc} never arises. The perfectness of $k_v$ is not needed in this
case.

Assume now that $[k_w:k_v]=\infty$; we shall produce
$\alpha\in\bK\setminus\oK$ with $W=V_{\alpha}$. One could adapt the arguments
given in \cite{p25} for $V=\Z_{(p)}$ by means of the notion of stacked sequence;
we give instead a shorter argument which uses minimal pairs and results of
\cite{APZAll}.

\vspace{.2cm}
\noindent
{\em Reduction to the case $K=\K$.}
Let $\widehat{K(X)}$ and $\widehat W$ be the completions of $K(X)$ and of $W$
with respect to $w$. The restriction of $w$ to $K$ is $v$, so the closure of $K$
in $\widehat{K(X)}$ is a completion of $K$ and $\K$ is canonically embedded in
$\widehat{K(X)}$. Moreover $X$ is transcendental over $\K$: otherwise $\K(X)$
would be a finite extension of $\K$, hence complete, and, containing the dense
subfield $K(X)$, it would coincide with $\widehat{K(X)}$ (see also
\cite[Lemma 2.4]{PerTransc}); its residue field would then be finite over $k_v$,
contradicting $[k_w:k_v]=\infty$. In particular the field $\K(X)$ of rational
functions over $\K$ is contained in $\widehat{K(X)}$. Put
$\widetilde W=\widehat W\cap\K(X)$. Then
$W=\widehat W\cap K(X)=\widetilde W\cap K(X)$, and $\widetilde W$ is a valuation
domain of $\K(X)$ with $\widetilde W\cap\K=\widehat V$. Since $V\subseteq
\widehat V$ and $W\subseteq\widehat W$ are immediate extensions of valuations, we
have $\Gamma_{\widetilde w}=\Gamma_w\cong\Z$ and
$k_{\widetilde w}=k_w$; hence $\widetilde W$ is a DVR of $\K(X)$, residually
algebraic over $\widehat V$, with residue field of infinite degree over the
residue field $k_v$ of $\widehat V$, and $k_v$ is perfect. Finally, the completion of
an algebraic closure of $\widehat{\K}=\K$ is again $\bK$. It is therefore enough
to prove the assertion for $\widetilde W$ over $\widehat V$, and we may and do
assume from now on that $K=\K$ is complete; note that $(K,v)$ then satisfies the
standing assumptions of Section \ref{sec:ram}, and that, $k_v$ being perfect, the separability
hypothesis of Lemma \ref{lem:ram} holds for every finite extension of $K$; 
no restriction on the characteristic is needed. In this situation
$\aK=\oK$ is an algebraic closure of $K$ and $\bK$ is its completion, 
so that the two notations of Section \ref{sec:notation} coincide from now on.

\vspace{.2cm}
\noindent
{\em The associated pseudo-convergent sequence.}
Let $\overline{W}$ be an extension of $W$ to $\aK(X)$, which exists by \cite[Ch. VI, $\S$ 8, Proposition 6]{Bourb} since 
$\aK(X)$ is algebraic over $K(X)$.  Since $v$ has a unique extension to $\aK$, 
the restriction of $\overline W$ to $\aK$ is necessarily $\aV$, so that $\overline W$ 
is a common extension of $\aV$ and of $W$. Then $\overline W$ is an immediate extension
of $\aV$: its residue field is algebraic over $\ok$, because $\aK(X)\mid K(X)$ is
algebraic and $k_w$ is algebraic over $k_v$, and $\ok$ is algebraically closed;
and $\Gamma_{\overline w}$ is a torsion extension of the divisible group
$\Gamma_{\overline v}$, hence equals it. The torsion assertion holds for the
same reason as the first: the extension $\aK(X)\mid K(X)$ being algebraic, 
$\Gamma_{\overline w}/\Gamma_w$ is torsion, and $\Gamma_w$ contains $\Gamma_v$ with finite index $e$; 
so $\Gamma_{\overline w}/\Gamma_v$ is torsion, 
and a fortiori so is $\Gamma_{\overline w}/\Gamma_{\overline v}$.

Following \cite[p. 287]{APZAll} we fix a strictly increasing sequence
$\{\lambda_n\}_{n\in\N}$ which is cofinal in
$\Lambda:=\{\overline{w}(X-a)\mid a\in \aK\}$. Such a sequence exists and may be
indexed by $\N$: indeed $\Lambda\subseteq\Gamma_{\overline v}\cong\Q$ is
countable, and $\Lambda$ has no maximum. The proof follows from the fact that $\overline W$ 
is immediate over $\overline V$ (see for example the proof of \cite[Theorem 1]{Kap}), 
for the sake of the reader we repeat here the argument: suppose that $\overline w(X-a),a \in\overline K$ were
maximal; choosing $c\in \aK$ with $v(c)=\overline w(X-a)$, which is possible
because $\Gamma_{\overline w}=\Gamma_{\overline v}=v(\aK^{\times})$, the residue
of $(X-a)/c$ lies in the residue field of $\overline W$, which is $\ok$, 
so it is the residue of some $b\in\aV$, and then $\overline w(X-a-bc)>\overline w(X-a)$, a contradiction. 
Being countable and without maximum, $\Lambda$ has cofinality $\omega$.

We can now associate to $\{\lambda_n\}$ a sequence of minimal pairs, as in \cite[p.~290, (21)]{APZAll}.
For $\lambda\in\Lambda$ set
$$d(\lambda)=\min\{[K(a):K]\mid a\in \aK,\ \overline w(X-a)\geq\lambda\},$$
which is a well-defined positive integer, and choose $t_m\in \aK$ with 
$\overline w(X-t_m)\geq\lambda_m$ and $[K(t_m):K]=d(\lambda_m)$. Since $\overline w(X-t_m)\geq\lambda_m$ and 
$\{\lambda_m\}$ is cofinal in $\Lambda$, the family $\{\overline w(X-t_m)\}_{m\in\N}$ is cofinal in $\Lambda$, as well; 
we may therefore choose indices $m(1)<m(2)<\cdots$ for which
\begin{equation}\label{X pseudo-limit}
    s_n:=t_{m(n)},\qquad \delta_n:=\overline w(X-s_n)
\end{equation}
give a strictly increasing sequence $\{\delta_n\}_{n\in\N}$, which is also cofinal
in $\Lambda$ as well. By construction
\begin{equation}\label{eq:minimal}
[K(s_n):K]=d(\lambda_{m(n)})\qquad\text{and}\qquad \lambda_{m(n)}\leq\delta_n .
\end{equation}

Each $(s_n,\delta_n)$ is a \emph{minimal pair} in the sense of \cite{APZAll}:
if $b\in \aK$ satisfies $[K(b):K]<[K(s_n):K]$ and $v(s_n-b)\geq\delta_n$, then
$$\overline w(X-b)\geq\min\{\overline w(X-s_n),v(s_n-b)\}\geq\delta_n
\geq\lambda_{m(n)}$$
by \eqref{eq:minimal}, so that $b$ competes in the definition of 
$d(\lambda_{m(n)})$ with a degree smaller than that of $s_n$, a contradiction.
We stress that the minimality must be taken with respect to the inequality
$\overline w(X-a)\geq\lambda_{m(n)}$, and not with respect to the equality
$\overline w(X-a)=\delta_n$: an element $b$ of smaller degree with
$\overline w(X-b)>\delta_n$ is not excluded by the latter, the set $\Lambda$ having
no maximum. This matters because the definition of a minimal pair $(a,\delta)$
itself involves an inequality, namely $[K(a):K]\leq[K(b):K]$ for every $b$ with
$v(b-a)\geq\delta$; the selection by the equality
$\overline w(X-a)=\delta$ is the one found in the literature (see
\cite[p.~290, (21)]{APZAll} and the proof of \cite[Corollary 3.2]{oke13}), and it does not
by itself yield a minimal pair.
From
$$v(s_{n+1}-s_n)=\overline w\big((X-s_n)-(X-s_{n+1})\big)
=\min\{\delta_n,\delta_{n+1}\}=\delta_n$$
it follows that $E=\{s_n\}_{n\in\N}$ is a pseudo-convergent sequence of $\aK$.
It is of \emph{transcendental type} over $\aK$ because $E$ does not have pseudo-limits in $\aK$ 
(this follows from \cite{Kap}, see also the comments on \cite[p. 7962]{PS20}): if $a\in\aK$ were a pseudo-limit of $E$, 
that is, if $v(a-s_n)=\delta_n$ for every $n$, then $\overline w(X-a)\geq\min\{\overline w(X-s_n),v(s_n-a)\}=\delta_n$ 
for every $n$, so that the element $\overline w(X-a)$ of $\Lambda$ would be an upper bound 
for a cofinal sequence of $\Lambda$, that is, a maximum of $\Lambda$, which we have just excluded. 
Furthermore,
$$\aV_E=\{\phi\in\aK(X)\mid \phi(s_n)\in\aV \text{ for all sufficiently large }n\in\N\}$$
is a valuation domain of $\aK(X)$ \cite[Theorem 2]{Kap}, \cite[Remark 4.2]{APZAll}, which is equal to $\overline W$. 
This follows immediately from \cite[Theorem 2]{Kap} since $X$ is a pseudo-limit of $E$ 
by \eqref{X pseudo-limit}, for the sake of the reader we give a short and direct argument: clearly,
it suffices to show that $\overline w(f)=v(f(s_n))$ 
for every non-zero $f\in\aK[X]$ and every large $n$. Write
$f=c\prod_{i}(X-r_i)$ with $c,r_i\in\aK$.  Each $\overline w(X-r_i)$ lies in $\Lambda$, 
so that  $\delta_n>\overline w(X-r_i)$ for $n$ large and for every $i$; 
for such $n$'s and $i$'s, 
we have $\overline w(X-r_i)=v(s_n-r_i)$. 
Summing the values yields $\overline w(f)=v(f(s_n))$ for every $n\in\N$ large enough, 
and the same equality extends to $\aK(X)$ by multiplicativity. Since $\phi\in\aV_E$ 
means precisely that $v(\phi(s_n))\geq0$ for $n$
large, the claim follows. Finally $W=\overline W\cap K(X)$.

\vspace{.2cm}
\noindent
{\em The denominators of the $\delta_n$ are bounded.}
For each $n\in\N$ let $v_n$ be the restriction of $v$ to $K(s_n)$ and let
$\Gamma_n,k_n$ be its value group and residue field, respectively. Since $W$ is residually
algebraic over $V$ and $[\Gamma_w:\Gamma_v]=e<\infty$, it is a residually algebraic
torsion extension of $V$, and $\{(s_n,\delta_n)\}$ is a sequence of minimal pairs
defining $w$ in the sense of \cite{APZAll}. The result we use is \cite[Theorem 2.3~b),~c)]{APZAll}, 
whose hypotheses are three.
First, the $(s_n,\delta_n)$ must be minimal pairs of definition of the $\overline
w_n:=\overline w_{s_n,\delta_n}$, which we have just verified. Secondly,
$\overline w=\sup_n\overline w_n$, which holds by \cite[Theorem 4.1]{APZAll}, since
$\overline w(X-s_n)=\delta_n$ and $\{\delta_n\}$ is cofinal in $\Lambda$. Thirdly,
$\overline w$ must not be a residually transcendental extension of $\overline v$; this is
clear, $\overline W$ being an immediate extension of $\aV$, so that $k_{\overline w}=\ok$.
The same statement, for several indeterminates and for systems indexed by an arbitrary
well-ordered set without last element, is \cite[Theorem 3.1]{oke13}; the present case is that
of one indeterminate, with $I=\N$. Both are proved for an arbitrary valuation of an arbitrary
field, with no separability and no characteristic assumption. In our notation it reads
\begin{align*}
n<m\Rightarrow \Gamma_n\subseteq\Gamma_m,\,\,k_n\subseteq k_m,\\
\Gamma_w=\bigcup_{n\in\N}\Gamma_n,\;\; k_w=\bigcup_{n\in\N}k_n;
\end{align*}
of these we shall only need the first line and the equality
$\Gamma_w=\bigcup_n\Gamma_n$.
We stress that we do not appeal to \cite[Theorem 5.1]{APZAll}, nor to
\cite[Corollary 3.2]{oke13}: in both, the system of minimal pairs is produced by the
selection under the equality $\overline w(X-a)=\delta$ discussed above, which does not by
itself yield a minimal pair, whereas the hypothesis of \cite[Theorem 2.3]{APZAll} and of
\cite[Theorem 3.1]{oke13} is that the pairs be minimal. The construction of
$\{(s_n,\delta_n)\}$ carried out above supplies exactly that hypothesis.
Since $W$ is a DVR, $\Gamma_w\cong\Z$, so there exists $n_0\in\N$ such that
$\Gamma_n=\Gamma_w$ for every $n\geq n_0$; in particular,
$e_{s_n}=[\Gamma_n:\Gamma_v]=e$ for every $n\geq n_0$, in the notation of
Section \ref{sec:eps}.
 Let $n\geq n_0$. The
element $\delta_n=v(s_{n+1}-s_n)$ lies in the value group of the restriction
$v_n'$ of $v$ to $K(s_n,s_{n+1})$. We apply Lemma \ref{lem:ram} with $F=K$,
$F_1=K(s_n)$ and $F_2=K(s_{n+1})$; its hypotheses hold, since $K$ is complete with respect
to a discrete valuation of rank one and the residue field extension of $K(s_n)\mid K$ is
separable, $k_v$ being perfect. Note that no separability is required of $K(s_n)\mid K$
itself, and, what is essential, no tameness: the extensions $K(s_n)\mid K$ might be wildly
ramified in general, which is why the lemma had to be proved without a tameness assumption.
It gives $e\big(K(s_n,s_{n+1})\mid K(s_{n+1})\big)\leq
e\big(K(s_n)\mid K\big)=e_{s_n}=e$; therefore
$$e(v_n'\mid v)=e\big(K(s_n,s_{n+1})\mid K(s_{n+1})\big)\cdot
e\big(K(s_{n+1})\mid K\big)\leq e\cdot e=e^2 .$$

Put $N=(e^2)!$. The group $\Gamma_{v_n'}$ is a subgroup of
$\Gamma_{\overline v}\cong\Q$ containing $\Gamma_v\cong\Z$ with index $m_n\leq e^2$, so that
$m_n\Gamma_{v_n'}\subseteq\Gamma_v$; as $m_n$ divides $N$, this gives
$\Gamma_{v_n'}\subseteq\frac1N\Gamma_v$ and in particular
$\delta_n\in\frac1N\Gamma_v$ for every $n\geq n_0$. The
strictly increasing sequence $\{\delta_n\}_{n\geq n_0}$ thus lies in the
infinite cyclic group $\frac1N\Gamma_v$, so it is unbounded, that is, cofinal in
$\Gamma_{\overline v}$. (We do not know whether the sharper bound
$\delta_n\in\frac1e\Gamma_v$ holds; only boundedness of the denominators is
needed here.)

\vspace{.2cm}
\noindent
{\em Conclusion.}
Since $v(s_{n+1}-s_n)=\delta_n$ tends to infinity, $E$ is a Cauchy sequence and
converges to a unique $\alpha\in\bK$; moreover $\alpha\notin\aK$, since a
pseudo-limit of $E$ in $\aK$ has been excluded above 
(i.e., $E$ is of transcendental type). For every non-zero
$\phi\in \aK(X)$ we have $\phi(s_n)\to\phi(\alpha)$ and $v(\phi(s_n))$ is
eventually constant, so that $v(\phi(\alpha))=v(\phi(s_n))$ for $n$ large;
therefore
$$\overline W=\oV_E=\{\phi\in \aK(X)\mid v(\phi(\alpha))\geq0\} =\oV_{\alpha},$$
whence $W=\overline W\cap K(X)=V_{\alpha}$. Going back to the initial setting, 
where $K$ is not necessarily complete,
this reads $\widetilde W=\{\phi\in\K(X)\mid v(\phi(\alpha))\geq0\}$ and therefore
$$W=\widetilde W\cap K(X)=\{\phi\in K(X)\mid v(\phi(\alpha))\geq0\}=V_{\alpha},$$
with $\alpha\in\bK\setminus\oK$; in particular $\alpha$ is transcendental over $K$. 
The sequence $\{s_n\} \subset\oK$ constructed above satisfies $\alpha=\lim_ns_n$ 
and $e_{s_n} =e=[\Gamma_w:\Gamma_v]$ for $n\geq n_0$, which is assertion (3).

It remains to prove assertion (2). If $[k_w:k_v]<\infty$, 
the first part of the proof gives some $\alpha_0\in\oK$ with $W=V_{\alpha_0}$, and no
$\beta\in\bK\setminus\oK$ can satisfy $W=V_{\beta}$, by Corollary \ref{lem:rank};
hence every $\alpha$ with $W=V_{\alpha}$ lies in $\oK$. If $[k_w:k_v]=\infty$ and
$\alpha\in\oK$ satisfied $W=V_{\alpha}$, then by Lemma \ref{lem:completion} the
completion of $K(X)$ with respect to $w$ would be $C_{\alpha}=\K(\alpha)$, a finite extension
of $\K$; completions being immediate, $k_w$ would be the residue field of $\K(\alpha)$ and
hence finite over $k_v$, a contradiction; hence every $\alpha$ with $W=V_{\alpha}$ lies in
$\bK\setminus\oK$.
\end{proof}

By Remark \ref{rem:Valpha}, for every $\alpha\in\bK$ transcendental over $K$, 
the valuation domain $V_{\alpha}$ is a residually algebraic torsion extension of $V$, 
and Theorem \ref{thm:main} says that the class of valuation domains $\{V_{\alpha}\mid \alpha\in\bK,\alpha \text{  transcendental over }K\}$ 
and the class of residually algebraic torsion extensions of $V$ have the same discrete members 
(i.e., DVRs); note that the former class is contained in the latter by Remark \ref{rem:Valpha}. These two classes
are not equal, however, and they fail to be so in both directions: not every $V_{\alpha}$ is discrete 
(Example \ref{ex:nondiscrete}), and not every residually algebraic torsion extension of $V$ is a $V_{\alpha}$ 
(Example \ref{prop:notValpha}).

\begin{remark}\label{rem:uniqueness}
The element $\alpha$ of Theorem \ref{thm:main} is unique up to the action of $G$, and in
particular the alternative in assertion (2) of the theorem does not depend on the
representative.
\end{remark}

\begin{proof}
Let $\alpha\in\bK$ be transcendental over
$K$ and let $\sigma\in G$, extended by continuity to $\bK$ as in (G2). Since
$\sigma$ fixes $K$ and $v\circ\sigma=v$, we have
$v(\phi(\sigma(\alpha)))=v(\sigma(\phi(\alpha)))=v(\phi(\alpha))$ for every
$\phi\in K(X)$, so that $V_{\sigma(\alpha)}=V_{\alpha}$. Conversely, let
$\alpha,\beta\in\bK$ be transcendental over $K$ with $V_{\alpha}=V_{\beta}$; by
Corollary \ref{lem:rank}, either both of them lie in $\oK$ or neither does. In the
first case $\alpha$ and $\beta$ are conjugate over $\K$ by
\cite[Theorem 3.2]{PerTransc}, hence $\beta=\sigma(\alpha)$ for some $\sigma\in G$
by (G1).

In the second case we first reduce to $K=\K$; we keep the notation $C_{\gamma}$ of
Lemma \ref{lem:completion}. Let $\widehat{K(X)}$ be the completion of
$K(X)$ with respect to the valuation associated with $W:=V_{\alpha}=V_{\beta}$, let
$\widehat W$ be its valuation ring, and identify $\K$ with the closure of $K$ in
it. Since $\alpha\notin\oK$, Lemma \ref{lem:completion}
gives $[\widehat{K(X)}:\K]=[C_{\alpha}:\K]=\infty$; hence $X$ is
transcendental over $\K$, for otherwise $\K(X)$ would be finite over $\K$, hence
complete, and, containing the dense subfield $K(X)$, would coincide with
$\widehat{K(X)}$. The subfield $\K(X)$ of $\widehat{K(X)}$ and the valuation
domain $\widetilde W=\widehat W\cap\K(X)$ are therefore defined intrinsically in
terms of $W$, while the identification of $\widehat{K(X)}$ with $C_{\alpha}$
(resp., $C_{\beta}$), carries $\K(X)$ onto $\K(\alpha)$ (resp., $\K(\beta)$), by
evaluation; hence
$$\{\phi\in\K(X)\mid v(\phi(\alpha))\geq0\}=\widetilde W=
\{\phi\in\K(X)\mid v(\phi(\beta))\geq0\},$$
which is the hypothesis with $\K$ in place of $K$.

Assume then $K=\K$. By Remark \ref{rem:conj} the rings $\overline{V_{\alpha}}$ and
$\overline{V_{\beta}}$ are two extensions of $V_{\alpha}$ to $\oK(X)$, hence
conjugate under $\Aut(\oK(X)\mid K(X))=G$; combining with \eqref{eq:conj} we get
$\overline{V_{\beta}}=\sigma(\overline{V_{\alpha}})
=\overline{V_{\sigma(\alpha)}}$ for some $\sigma\in G$.
Writing $\gamma=\sigma(\alpha)$, the equality of these two valuation domains gives
$v(\beta-s)=v(\gamma-s)$ for every $s\in\oK$, both sides being the value of $X-s$.
Since $\oK$ is dense in $\bK$ we may choose $s\in\oK$ with $v(\beta-s)$
arbitrarily large, and then
$v(\beta-\gamma)\geq\min\{v(\beta-s),v(\gamma-s)\}=v(\beta-s)$ is arbitrarily
large; hence $\beta=\gamma=\sigma(\alpha)$.
\end{proof}

\subsection{\texorpdfstring{When each of the two cases occurs}{When each of the two cases occurs}}\label{subsec:cases}

The first case of Theorem \ref{thm:main} is governed by the completeness of $K$
and the second by the size of $\ok$ over $k_v$; both can fail at once, and the
class is then empty.  Examples \ref{ex:CC} and \ref{ex:Qps} show that neither of
the two hypotheses of Corollary \ref{cor:vacuous} can be dropped.

\begin{proposition}\label{rem:byproduct}
Let $V$, $K$ and $W$ be as in Theorem \ref{thm:main} and read $X$ inside the completion
$\widehat{K(X)}$ of $K(X)$, which contains a canonical copy of $\K$. Then
$$X \text{ is algebraic over } \K \iff [k_w:k_v]<\infty .$$
In particular, if $K$ is complete, then $[k_w:k_v]=\infty$ for every such $W$, and hence
$\alpha\notin\oK$ for every $\alpha$ with $W=V_{\alpha}$.
\end{proposition}

\begin{proof}
If $[k_w:k_v]<\infty$, then $W=V_{\alpha}$ for some $\alpha\in\oK$ by Theorem
\ref{thm:main}\,(1) and (2), and $\phi\mapsto\phi(\alpha)$ is an isomorphism of valued
fields $K(X)\to K(\alpha)$;
it therefore identifies $\widehat{K(X)}$ with the completion $\K(\alpha)$ of
$K(\alpha)$, which is a finite extension of $\K$, and it carries $X$ to $\alpha$,
an element algebraic over $\K$. Conversely, if $[k_w:k_v]=\infty$, the reduction
step in the proof of Theorem \ref{thm:main} shows that $X$ is transcendental over
$\K$. For the last assertion, note that if $K=\K$ then $X$ is transcendental over
$\K$ by definition.
\end{proof}

\begin{corollary}\label{cor:vacuous}
Let $V$ be a complete DVR whose residue field is algebraically closed or real closed. Then
there is no extension of $V$ to $K(X)$ which is a DVR and residually algebraic
over $V$.
\end{corollary}

\begin{proof}
Such a residue field $k_v$ is perfect and satisfies $[\ok:k_v]\leq2$
by the theorem of Artin and Schreier \cite[Satz~4]{AS27}
(or see \cite[Ch.~VI, Corollary 9.3]{Lang}). If $W$ were such
an extension, Proposition \ref{rem:byproduct} would give $[k_w:k_v]=\infty$, whereas $k_w$
embeds into $\ok$ over $k_v$ and hence $[k_w:k_v]\leq2$.
\end{proof}

Completeness cannot be omitted in Corollary \ref{cor:vacuous}: if $K\neq\K$, any $\alpha\in\oK$
transcendental over $K$ belongs to the class, its value group $v(K(\alpha)^{\times})$ lying
between $\Gamma_v\cong\Z$ and $v(\K(\alpha)^{\times})$, which contains $\Gamma_v$ with finite
index, so that $V_{\alpha}$ is a DVR; and it is residually algebraic over $V$
by Remark \ref{rem:Valpha}.

\begin{remark}\label{rem:byproductdisc}
{\em That alternative (2) of Theorem \ref{thm:main} degenerates over a complete $K$ is
consistent with the fact that no element of $\oK$ is then transcendental over $K$, and it
explains why the reduction step of the proof is needed only when $[k_w:k_v]=\infty$.
Proposition \ref{rem:byproduct} also
refines the observation contained in the proof of
\cite[Proposition 2.24]{p25} in the case $K=\Q$. For the provenance of the
dichotomy itself, see Remark \ref{rem:IZ}.} \eoe
\end{remark}

The case $[k_w:k_v]=\infty$ of Theorem \ref{thm:main} is not empty, and this is so for every DVR whose residue field
admits finite extensions of arbitrary large degree. In particular, it is not empty in equal characteristic $p$,
where $\oK\mid\K$ need not be separable.

\begin{proposition}\label{prop:nonvacuous}
Let $V$ be a DVR with perfect residue field $k_v$. The following
statements are equivalent.
\begin{enumerate}
\item $[\ok:k_v]=\infty$.
\item There is $\alpha\in\bK\setminus\oK$ --- necessarily transcendental over $K$, so that
$V_{\alpha}$ is defined --- with $V_{\alpha}$ a DVR.
\item There is an extension $W$ of $V$ to $K(X)$ which is a DVR and
residually algebraic over $V$, with $[k_w:k_v]=\infty$.
\end{enumerate}
These conditions hold precisely when $k_v$ is neither algebraically closed nor real closed; in
particular they hold whenever $k_v$ is finite, and whenever $k_v$ is a number field.
\end{proposition}

\begin{proof}
$(1)\Rightarrow(2)$. Normalize $v$ so that $v(\pi)=1$, where $\pi$ is the generator of the maximal ideal
of $V$ fixed in Section \ref{sec:notation}; recall that $\pi$ is also a uniformizer of
$\widehat V$. Denote by $\K^{\mathrm{ur}}\subseteq\oK$ the maximal unramified extension of
$\K$, that is, the union of the finite unramified extensions of $\K$ inside
$\oK$; its value group is $\Gamma_v$ and its residue field is $\ok$.

Since $k_v$ is perfect, $\ok\mid k_v$ is separable, and $[\ok:k_v]=\infty$; hence
$\ok\mid k_v$ has finite subextensions of arbitrarily large degree, each of them
simple. Choose inductively elements $\overline{a_n}\in\ok$, $n\geq1$, with
$$[k_v(\overline{a_n}):k_v]>(n-1)\cdot D_{n-1},\qquad
D_{n}:=[k_v(\overline{a_1},\ldots,\overline{a_n}):k_v]$$
(and $D_0=1$). Each $\overline{a_n}$ is separable over $k_v$, so there is a finite
unramified extension $M^{(n)}$ of $\K$ inside $\oK$ with residue field
$k_v(\overline{a_n})$, and the residue map $O_{M^{(n)}}\to k_v(\overline{a_n})$ is
surjective: choose $a_n\in O_{M^{(n)}}$ lifting $\overline{a_n}$. Then
$a_n\in\K^{\mathrm{ur}}$, and $a_n$ is a unit, $\overline{a_n}$ being non-zero.
Put
$$s_N=\sum_{n=1}^{N}a_n\pi^{n},\qquad \alpha=\sum_{n=1}^{\infty}a_n\pi^{n},$$
the series converging in $\bK$ because $v(a_n\pi^{n})=n$; thus
$s_N\in\K(a_1,\ldots,a_N)\subseteq\K^{\mathrm{ur}}$ and $v(\alpha-s_N)=N+1$,
the field $\K(a_1,\ldots,a_N)$ being unramified over $\K$ as a compositum of
finitely many unramified extensions.

\vspace{.2cm}
\noindent
{\bf Claim 1.} $\alpha\notin\oK$. The sequence $\{s_N\}$ is Cauchy and contained
in $\K^{\mathrm{ur}}$, so $\alpha$ lies in the closure of $\K^{\mathrm{ur}}$ in $\bK$, whose
value group is again $\Gamma_v$. Suppose $\alpha\in\oK$. Then
$e(\K(\alpha)\mid\K)=1$; the residue field extension of $\K(\alpha)\mid\K$ is
separable, $k_v$ being perfect; hence $\K(\alpha)\mid\K$ is unramified and
$\alpha\in\K^{\mathrm{ur}}$. Let then $M$ be a finite unramified extension of $\K$
containing $\alpha$ and put $c=[k_M:k_v]$. For every $N$ the field
$M_N=M\K(a_1,\ldots,a_N)$ is unramified over $\K$ with residue field
$k_M\cdot k_v(\overline{a_1},\ldots,\overline{a_N})$, of degree at most $c\,D_N$
over $k_v$, and it contains both $\alpha$ and $s_N$. Hence
$(\alpha-s_N)/\pi^{N+1}$ is a unit of $M_N$, and its residue is
$\overline{a_{N+1}}$, so that
$$[k_v(\overline{a_{N+1}}):k_v]\leq c\,D_N .$$
For $N\geq c$ this contradicts the choice $[k_v(\overline{a_{N+1}}):k_v]>N\,D_N$.
Therefore $\alpha\notin\oK$, and $\alpha$ is transcendental over $K$, every
element of $\bK$ algebraic over $K$ being algebraic over $\K$ and hence lying in
$\oK$. In particular $V_{\alpha}$ is defined.

\vspace{.2cm}
\noindent
{\bf Claim 2.} {\em $V_{\alpha}$ is a DVR.} 
Since $\alpha$ lies in the completion of $\K^{\ur}$, the proof follows 
at once from Proposition \ref{prop:criterion} below; we give here a different and direct argument.
Each $s_N$ lies in the unramified
extension $\K(a_1,\ldots,a_N)$ of $\K$, whose value group is $\Gamma_v$; hence
$v(f(s_N))\in\Gamma_v$ for every $f\in K[X]$. Since $s_N\to\alpha$ and
$f(\alpha)\neq0$ for $f\neq0$, we have $f(s_N)\to f(\alpha)$ and therefore
$v(f(\alpha))=v(f(s_N))\in\Gamma_v$ for $N$ large. The value group
$v(K(\alpha)^{\times})$ of $V_{\alpha}$ is thus contained in $\Gamma_v$ and
contains it, so it equals $\Gamma_v\cong\Z$; hence $V_{\alpha}$ is a DVR.

$(2)\Rightarrow(3)$. Put $W=V_{\alpha}$. It is residually algebraic over $V$ by
Remark \ref{rem:Valpha} and discrete of rank one by hypothesis, and $[k_w:k_v]=\infty$ by
assertion (2) of Theorem \ref{thm:main}, since $\alpha\notin\oK$.

$(3)\Rightarrow(1)$. The residue field $k_w$ is algebraic over $k_v$, hence admits a
$k_v$-embedding into $\ok$; therefore $[\ok:k_v]\geq[k_w:k_v]=\infty$.

Finally, by the theorem of Artin and Schreier \cite[Satz~4]{AS27}
(or see \cite[Ch.~VI, Corollary 9.3]{Lang}), $[\ok:k_v]<\infty$ holds only when $\ok=k_v$ or $[\ok:k_v]=2$, the latter exactly when $k_v$ is real closed.
\end{proof}

The two cases that follow are instances of Proposition \ref{prop:nonvacuous}.
They bear on Corollary \ref{cor:vacuous} as well: in both of them $K$ is complete, so that
completeness alone does not empty the class, and what makes the class non-empty is that the
residue field $\mathbb{F}_p$ has an algebraic closure of infinite degree over it.

\begin{example}\label{ex:nonvacuous}
In the proof of Proposition \ref{prop:nonvacuous} one chooses elements
$\overline{a_n}\in\ok$ with $[k_v(\overline{a_n}):k_v]>(n-1)D_{n-1}$, where
$D_n=[k_v(\overline{a_1},\ldots,\overline{a_n}):k_v]$ and $D_0=1$. We give two instances in
which the $\overline{a_n}$ may be taken to generate $\mathbb{F}_{p^{n!}}$ over
$\mathbb{F}_p$; in both, $\zeta_n$ denotes a primitive $(p^{n!}-1)$-th root of unity in
$\oK$, so that $\zeta_n$ is a unit whose residue generates
$\mathbb{F}_{p^{n!}}$ (Section \ref{sec:notation}, roots of unity of order prime to
$p$; here $m=p^{n!}-1$).
\begin{itemize}
\item $V=\Z_p$, $K=\K=\Q_p$, $\oK=\overline{\Q_p}$, $\bK=\mathbb{C}_p$,
$\pi=p$:
$$\alpha=\sum_{n\geq1}\zeta_np^{n}\in\mathbb{C}_p .$$
\item $V=\mathbb{F}_p[[t]]$, $K=\K=\mathbb{F}_p((t))$, $\pi=t$:
$$\alpha=\sum_{n\geq1}\zeta_nt^{n}\in\bK .$$
Here $\K$ is not perfect, so that $\oK\mid\K$ is not separable and
$G=\Aut(\oK\mid\K)$ is not a Galois group: this is the situation which is not
covered by \cite[Corollary 2.28]{p25}.
\end{itemize}
In both cases $D_N=[\mathbb{F}_{p^{N!}}:\mathbb{F}_p]=N!$ and
$[k_v(\overline{a_{N+1}}):k_v]=(N+1)!>N\cdot N!$, as required.
\end{example}

\begin{remark}\label{rem:overrings}
{\em Theorem \ref{thm:main} applies in particular to the extensions $W$ of $V$ to
$K(X)$ which are DVRs containing the polynomial ring $V[X]$: each such $W$ that
is residually algebraic over $V$ equals $V_{\alpha}$ for an $\alpha\in\bK$
transcendental over $K$ with $v(\alpha)\geq0$, the inequality being equivalent to
$X\in W$. Comparing with the description recalled at the end of Section
\ref{sec:notation} and with \eqref{eq:trace}, such a $W$ is, 
when $\alpha\in\oK$, $\oV_{\alpha,\infty}\cap K(X)$, the rank two valuation domain $\oV_{\alpha,\infty}$ 
being the one of Section \ref{sec:monomial},
whereas the residually transcendental extensions of $V$ containing $V[X]$ are the contractions to $K(X)$
of the monomial valuation domains $\oV_{\alpha,\delta}$ with $\delta$ finite. For $V=\Z_{(p)}$
the valuation overrings of $V[X]$ lying over $V$ are
among the objects ordered, by means of MacLane's construction, by Loper and Tartarone
\cite{LT09} in their study of the integrally closed rings lying between $\Z[X]$ and
$\Q[X]$; Theorem \ref{thm:main} identifies those of them that are DVRs and
residually algebraic over $V$.}
\end{remark}

\subsection{\texorpdfstring{Two families of examples}{Two families of examples}}\label{subsec:families}

The two families that follow are the equal characteristic realizations of the
hypotheses of Theorem \ref{thm:main}; every object considered so far can be
written down explicitly in them, and we shall use them, here and in
\S\ref{sec:examples}, to illustrate the results of the paper.

\begin{example}\label{ex:families}
Let $F$ be a field and consider
\begin{itemize}
\item[(A)] $V=F[[t]]$, $K=F((t))$, with $v$ the $t$-adic valuation and $\pi=t$;
\item[(B)] $V=F[t]_{(f)}$, $K=F(t)$, with $f\in F[t]$ irreducible and $\pi=f$.
\end{itemize}
Then the residue field of $v$ is $k_v=F$ in case (A) and $k_v=k:=F[t]/(f)$ in case (B), so that in both
cases $k_v$ is perfect precisely when $F$ is; moreover
$$\text{(A)}\ \ \K=K,\qquad
\text{(B)}\ \ \widehat V\cong k[[z]],\ \ \K\cong k((z)),\ \ z=f(t),\ \
K\neq\K .$$
Thus the completion of (B) is an instance of (A), and the two families differ
exactly by completeness --- which, by Proposition \ref{rem:byproduct}, is what
decides between the two cases of Theorem \ref{thm:main}.  If $F$ is perfect,
writing $\kappa=k_v$ and $z$ for the uniformizer in either case, the maximal
unramified extension is 
$$\K^{\ur}=\bigcup_{\kappa'}\kappa'((z)),\qquad
\widehat{\K^{\ur}}=\overline{\kappa}((z)),$$
$\kappa'$ running over the finite extensions of $\kappa$ inside
$\overline\kappa$.
\end{example}

\begin{proof}
In case (B) the complete local ring $\widehat V$ contains a field, hence admits
a coefficient field $k\subseteq\widehat V$ by Cohen's structure theorem (\cite{Cohen}; for
$F$ perfect, $k\mid F$ is finite separable and one may instead lift a primitive
element by Hensel's lemma, as in Step 1 of the proof of Lemma \ref{lem:ram});
expanding in powers of $f$ then gives $\widehat V\cong k[[z]]$.  That
$K\neq\K$ follows from $k((z))$ having infinite transcendence degree over $k$,
whereas $K$ has transcendence degree $1$ over $F$ and $k\mid F$ is algebraic.

For the second display, $\kappa'((z))\mid\kappa((z))$ has $e=1$ 
(because the uniformizer $z$ is the same in both power series fields) and separable
residue field extension, hence is unramified; conversely a finite unramified
$M\mid\K$ has residue field some finite separable $\ell\mid\kappa$, and Step 1
of the proof of Lemma \ref{lem:ram} gives $\theta\in O_M$ with
$M=\K(\theta)\cong\ell((z))$, where the last isomorphism follows as above by Cohen's structure theorem.  
Finally the union is dense in the complete field
$\overline\kappa((z))$, the truncations of a series over $\overline\kappa$
having coefficients in a finite extension of $\kappa$.
\end{proof}

The two cases of Theorem \ref{thm:main} are distributed between the families in
the sharpest possible way, as the following pair of rings shows.

\begin{example}\label{ex:CC}
Take $F=\C$ and compare $V=\C[[t]]$, of type (A), with $V'=\C[t]_{(t)}$, of type
(B) with $f=t$; the two have the same residue field $\C$ and the same completion
$\C((t))$.  Then there is \emph{no} extension of $V$ to $\C((t))(X)$ which is a
DVR and residually algebraic over $V$, whereas the extensions of $V'$ to
$\C(t)(X)$ which are DVRs and residually algebraic over $V'$ are exactly the
$V_{\alpha}$ for the Puiseux series $\alpha$ over $\C$ transcendental over
$\C(t)$, all of them with $k_w=k_v=\C$.  For instance
$$\alpha=\sum_{n\geq1}t^{\,n!},\quad [\Gamma_w:\Gamma_v]=1,\qquad
\beta=\sum_{n\geq1}t^{\,n!/2},\quad [\Gamma_w:\Gamma_v]=2,$$
where $w$ is the valuation associated to $V_{\alpha},V_{\beta}$, respectively.
The first of them gives $W=\{\phi\in\C(t)(X)\mid\phi(\alpha)\in\C[[t]]\}$, the
construction over $\C[t]_{(t)}\subset\C[[t]]$ that corresponds, in equal
characteristic, to the one over $\Z_{(p)}\subset\Z_p$.
\end{example}

\begin{proof}
The first assertion is Corollary \ref{cor:vacuous}, the ring $\C[[t]]$ being
complete with algebraically closed residue field.  For the second, 
let $W$ be such an extension of $V'$  to $\C(t)(X)$; 
its residue field embeds into $\ok=\C$ over
$k_v=\C$, so that $k_w=\C$,  $[k_w:k_v]=1$, and Theorem \ref{thm:main} gives
$W=V_{\alpha}$ for an $\alpha\in\oK$ transcendental over $K$.  Conversely, every
such $\alpha$ yields a DVR, as observed after Corollary \ref{cor:vacuous}; and
$\oK$, the algebraic closure of $\C((t))$, is the field of Puiseux series over
$\C$ by the theorem of Newton and Puiseux (see for example \cite[Ch. IV, \S 2, Proposition 8]{Ser}).

 Let $a_n=1$ if $n=m!$ for some $m\ge1$ and $a_n=0$ otherwise, so that
$\alpha = \sum_{n\ge0}a_nt^n$.
Assume to the contrary that $\alpha$ is algebraic over
$\mathbb{C}(t)$. An algebraic power series is $D$-finite \cite[Theorem 6.4.6]{Sta},
so the sequence $\{a_n\}$ is $P$-recursive \cite[Proposition 6.4.3]{Sta} and, hence there are
 polynomials $p_0,\dots,p_r\in\mathbb{C}[y]$, for some $r> 0$, with $p_r\neq0$, such that
$p_r(n)a_{n+r}+\dots+p_0(n)a_n=0$ for every $n\ge0$. 
Let $N$ be an integer such that $p_r(n) \neq 0$ for all integers $n \geq N$
and choose $m$ with $m!\ge N$ and $(m+1)!-m!>r$. Then
$a_{m!+1}=\dots=a_{m!+r}=0$, and applying the recurrence successively at
$n=m!+1,m!+2,\dots$, where $p_r(n)\neq0$, gives $a_n=0$ for every $n>m!$,
against $a_{(m+1)!}=1$. Hence $\alpha$ is transcendental over $\mathbb{C}(t)$. 
The same argument shows that $\beta$ is transcendental over $\mathbb{C}(t^{1/2})$ and a fortiori over $\mathbb{C}(t)$.
\end{proof}

\begin{example}\label{ex:Qps}
The second bullet of Example \ref{ex:nonvacuous} is of type (A) of Example
\ref{ex:families}, with $F=\mathbb{F}_p$; the first is not, $\Z_p$ being of
mixed characteristic.  In characteristic zero one may take $V=\Q[[t]]$ and
$$\alpha=\sum_{n\geq1}2^{1/n!}\,t^{n}\in\bK\setminus\oK ,$$
the check being that of Example \ref{ex:nonvacuous}, since
$\Q(2^{1/1!},\ldots,2^{1/N!})=\Q(2^{1/N!})$ has degree $N!$ over $\Q$.  Here
$\Gamma_w=\Gamma_v=\Z$, and $[k_w:k_v]=\infty$ by Theorem \ref{thm:main}\,(2),
$\alpha$ not lying in $\oK$.
\end{example}

Taken together the two examples delimit Corollary \ref{cor:vacuous} from both
sides, and independently of one another: Example \ref{ex:CC} shows that the
hypothesis of completeness cannot be dropped, the class over $\C[t]_{(t)}$
being infinite while the one over the complete $\C[[t]]$ is empty, and Example
\ref{ex:Qps} shows the same of the hypothesis on the residue field, over a base which is complete.

\section{The hypothesis of discreteness}\label{sec:examples}

We now establish the two failures recorded after the proof of Theorem
\ref{thm:main}, in the following two examples.
For $\alpha\in\oK$ transcendental over $K$, the domain $V_{\alpha}$ is a DVR, as
was observed after Corollary \ref{cor:vacuous}; for $\alpha\in\bK\setminus\oK$, that argument
breaks down, and discreteness may indeed fail; about this fact see also \cite[Remark 2.18 \& Theorem 2.26]{p25}.
The example that follows exhibits, over
$\Q_p$, an $\alpha$ for which the value group of $V_{\alpha}$ contains elements of
arbitrarily large denominator and is therefore not cyclic.

\begin{example}\label{ex:nondiscrete}
Let $p \in\Z$ be a prime, $V=\Z_p$, $K=\Q_p$, so that $\K=K$, $\oK=\overline{\Q_p}$ and
$\bK=\mathbb{C}_p$, $\pi=p$, and normalize $v$ so that $v(p)=1$. Fix a prime
$\ell\neq p$ and a compatible system $\{\theta_n\}_{n\geq1}\subset\oQp$ of
$\ell$-power roots of $p$, that is, $\theta_n^{\ell^n}=p$ and
$\theta_n^{\ell}=\theta_{n-1}$ for every $n\geq1$, where $\theta_0=p$,
 so that $v(\theta_n) = \ell^{-n}$ and 
 $\theta_n = (\theta_N)^{\ell^{N-n}} \in \Q_p(\theta_N)$ for every $n \leq N$ and let
$$u_n=p^n\theta_n,\qquad s_N=\sum_{n=1}^{N}u_n,\qquad
\alpha=\sum_{n=1}^{\infty}u_n .$$
By our assumption, we have $s_N\in\Q_p(\theta_N)$ for every $N\in\N$.
Since $v(u_n)=n+\ell^{-n}$ is strictly increasing and tends to infinity, 
the series converges in $\C_p$, that is $\alpha\in\C_p$. 
Then $\alpha$ is transcendental over $\Q_p$, the valuation domain $\Z_{p,\alpha}$ 
is a residually algebraic torsion extension of $\Z_p$ which is not
discrete, and $\Q_p(s_N)=\Q_p(\theta_N)$ for every $N$, so that $e_{s_N}=\ell^{N}$  are unbounded.
\end{example}

\begin{proof}
We have $v(\alpha-s_N)=v(u_{N+1})=N+1+\ell^{-(N+1)}$.
We claim that the value group $\Gamma=v(\Q_p(\alpha)^{\times})$ of $\Z_{p,\alpha}$ is
not discrete. Note first that this forces $\alpha$ to be transcendental over $\Q_p$,
because a finite extension of $\Q_p$  has a discrete value group;
in particular $\Z_{p,\alpha}$ is defined, and it is then a non-discrete torsion extension of $\Z_p$,
and it is also residually algebraic over $\Z_p$ by Remark \ref{rem:Valpha}.

 As $\ell \neq p$, we have $\tau(\theta_n) = \zeta\theta_n$ for some $\ell$-power root of unity $\zeta \neq 1$
whenever $\tau$ is a $\Q_p$-embedding of $\Q_p(\theta_N)$
into $\oQp$ with $\tau(\theta_n)\neq\theta_n$ (indeed
$\theta_n=(\theta_N)^{\ell^{N-n}}$ lies in $\Q_p(\theta_N)$, so that
$\tau(\theta_n)$ is defined; and
$\tau(\theta_n)^{\ell^{n}}=\tau(p)=p=\theta_n^{\ell^{n}}$, whence
$\zeta^{\ell^{n}}=1$), and
$v(1- \zeta) = 0$ by Section \ref{subsec:roots}; hence, for every $n\leq N$,
$v(\theta_n - \tau(\theta_n)) = v(\theta_n) = \ell^{-n}$.
Let $\tau\neq\mathrm{id}$ be such an embedding, so  $\tau(\theta_N)\neq\theta_N$. 
In particular, the set $\{n\leq N\mid\tau(\theta_n)\neq\theta_n\}$ is non-empty;
the values $n+\ell^{-n}$ are pairwise distinct, so we obtain
\begin{equation}\label{eq:ex51}
v(s_N-\tau(s_N))=\min\{n+\ell^{-n}\mid n\leq N, \tau(\theta_n)\neq\theta_n\}\leq N+\ell^{-N} .
\end{equation}
Recall that $\Q_p(s_N)\subseteq  \Q_p(\theta_N)$.
In particular $\tau(s_N)\neq s_N$, so that $\Q_p(\theta_N)$ admits no non-trivial $\Q_p(s_N)$-embedding
and therefore $\Q_p(s_N)=\Q_p(\theta_N)$. As $\Q_p(\theta_N)\mid \Q_p$ is 
totally ramified of degree $\ell^{N}$, we get $e_{s_N}=\ell^{N}$.

Let now $m_N\in \Q_p[X]$ be the minimal polynomial of $s_N$ and let $\sigma$
run over the distinct $\Q_p$-embeddings of $\Q_p(s_N)=\Q_p(\theta_N)$ into $\oQp$, so that
$v(m_N(\alpha))=\sum_{\sigma}v(\alpha-\sigma(s_N))$. For $\sigma\neq\mathrm{id}$,
\eqref{eq:ex51} gives $v(s_N-\sigma(s_N))\leq N+\ell^{-N}<v(\alpha-s_N)$,
so that $v(\alpha-\sigma(s_N))=v(s_N-\sigma(s_N))$, a value of the form
$n+\ell^{-n}$ with $n\leq N$ and therefore lying in $\ell^{-N}\Z$. Consequently
$$v(m_N(\alpha))=\Big(N+1+\frac{1}{\ell^{N+1}}\Big)+\gamma_N \quad\text{with}\quad
\gamma_N\in \ell^{-N}\Z,$$
an element of $\Gamma$ whose denominator is exactly $\ell^{N+1}$. As $N\in\N$ is
arbitrary, $\Gamma$ is not contained in $\ell^{-N}\Z$ for any $N$, hence it is
not cyclic; that is, $\Z_{p,\alpha}$ is not discrete, as claimed.
\end{proof}

Example \ref{ex:nondiscrete} settles the first of the two failures but not the
second. The valuation domain it produces is a $V_{\alpha}$, and by Remark \ref{rem:Valpha}
it is also a residually algebraic torsion extension of $V$; it therefore lies in both
classes and does not separate them. What is needed is a residually algebraic torsion
extension of $V$ which is not of the form $V_{\alpha}$, and the next example produces one.

\begin{example}\label{prop:notValpha}
Let $p$ be a prime.
There exists a torsion extension $W$ of
$\Z_p$ to $\Q_p(X)$ which is residually algebraic over $\Z_p$ and which is not of the form
$\Z_{p,\alpha}$ for any $\alpha\in \C_p$ transcendental over $\Q_p$. Such a $W$ is not
a DVR; in particular, the hypothesis that $W$ be a DVR cannot be omitted in
Theorem \ref{thm:main}.
\end{example}

\begin{proof}
The field $\mathbb{C}_p$ is not spherically complete
\cite[$\S$ 3.4]{Rob}, so that there is a nested sequence of balls
$B_n=\{x\in\mathbb{C}_p\mid v(x-c_n)\geq\delta_n\}$, with $\{\delta_n\}\subset\Q$ strictly increasing,
whose intersection is empty. Then the sequence $\{c_n\}$ is pseudo-convergent with
$v(c_{n+1}-c_n)=\delta_n$, and its pseudo-limits are exactly the points of $\bigcap_nB_n$;
hence it has none in $\mathbb{C}_p$. Note that $\{\delta_n\}$ is necessarily bounded: were it
unbounded, $\{c_n\}$ would be a Cauchy sequence and its limit in the complete field
$\mathbb{C}_p$ would lie in every $B_n$. By density, we can choose
 elements $s_n\in\oQp$ with $v(s_n-c_n)>\delta_n$: then $E=\{s_n\}$ 
 is a pseudo-convergent sequence of $\oQp$ with $v(s_{n+1}-s_n)=\delta_n$, 
 and $E$ has the same pseudo-limits as $\{c_n\}$, hence none in $\mathbb{C}_p$.

Since $\oQp$ is algebraically closed, a pseudo-convergent sequence of algebraic
type over $\oQp$ admits a pseudo-limit in $\oQp$ \cite[Theorem 3]{Kap}; therefore $E$ is of
transcendental type over $\oQp$ and
$$\overline W=\oZp_E=\{\phi\in \oQp(X)\mid \phi(s_n)\in \oZp 
\text{ for all sufficiently large } n\}$$
is a valuation domain of $\oQp(X)$, an immediate extension of $\oZp$, with
$\overline w(X-s_n)=\delta_n$ for every $n$ (\cite[Theorem 2]{Kap}). The set
$\Lambda=\{\overline w(X-a)\mid a\in\oQp\}$ is bounded: for $a\in\oQp$ one has
$\overline w(X-a)=v(a-s_n)$ for all large $n$, a value which is eventually constant 
and smaller than $\delta_n$, since $a$ is not a pseudo-limit of $E$; hence $\Lambda$ is bounded above by $\sup_n\delta_n<\infty$. 
This is precisely the configuration discussed in Remark
\ref{rem:IZ}. Set $W=\overline W\cap \Q_p(X)$. Then $W\cap \Q_p=\Z_p$ and, 
the extension $\overline W\mid\oZp$ being immediate, $\Gamma_w\subseteq\Gamma_{\overline w}=\Gamma_{\overline v}=\Q$ 
and $k_w\subseteq k_{\overline w}=\ok=\overline{\mathbb{F}_p}$;
hence $W$ is a residually algebraic torsion extension of $V$.

Suppose that $W=\Z_{p,\alpha}$ for some $\alpha\in\mathbb{C}_p$ transcendental over $\Q_p$. Then $\alpha\notin\oQp$,
so that by
Remark \ref{rem:conj} the valuation domain $\overline{\Z_{p,\alpha}}$ of $\oQp(X)$ is
defined and restricts to $\Z_{p,\alpha}=W$ on $\Q_p(X)$.
 Since $\mathbb{Q}_p$ is complete, $\overline{\mathbb{Q}_p} \mid \mathbb{Q}_p$ is algebraic,
hence so is $\overline{\mathbb{Q}_p}(X) \mid \mathbb{Q}_p(X)$, and the latter is
quasi-Galois (indeed Galois here, $\operatorname{char} \mathbb{Q}_p = 0$; only the
transitivity (G1), which needs no separability, is used). Therefore
$\overline W$ and $\overline{\Z_{p,\alpha}}$, being two extensions of $W$ to
it, are conjugate: $\overline W=\sigma(\overline{\Z_{p,\alpha}})$ for some
$\sigma\in\Aut(\oQp(X)\mid \Q_p(X))$, by (G1) of Section \ref{sec:notation}. By
\eqref{eq:conj}, $\overline W=\overline{\Z_{p,\sigma(\alpha)}}$. Writing $\beta=\sigma(\alpha)$, we obtain
$$\delta_n=\overline w(X-s_n)=v(\beta-s_n)\qquad \text{for every } n\in\N,$$
that is, $\beta$ is a pseudo-limit of $E$ in $\mathbb{C}_p$, against the choice of $E$. 
Therefore $W\neq \Z_{p,\alpha}$ for every $\alpha$, and $W$ is not discrete by Theorem \ref{thm:main}.
\end{proof}

\begin{corollary}\label{cor:notemb}
Let $W$ be the extension constructed in Example \ref{prop:notValpha} and let
$\widehat{\Q_p(X)}$ be the completion of $\Q_p(X)$ with respect to $W$. Then $\widehat{\Q_p(X)}$
admits no isometric $\Q_p$-embedding into  $\C_p$; equivalently, it is not isomorphic, as a valued
field over  $\Q_p$, to any complete field $L$ with  $\Q_p\subseteq L\subseteq\C_p$.
\end{corollary}

\begin{proof}
Suppose  $\iota\colon\widehat{\Q_p(X)}\to\C_p$ were an isometric $\Q_p$-embedding and put $x=\iota(X)$. 
Then $x$ is transcendental over $\Q_p$, $\iota$ being injective, and
$w(\phi)=v(\iota(\phi))=v(\phi(x))$ for every $\phi\in \Q_p(X)$; hence $W=V_x$, against Example
\ref{prop:notValpha}. Conversely an isomorphism onto a complete field $L$ with $\Q_p\subseteq
L\subseteq\C_p$, followed by the inclusion of $L$ in  $\C_p$, would be such an embedding.
\end{proof}

\begin{remark}\label{rem:IZ}
{\em Corollary \ref{cor:notemb} contradicts \cite[Proposition 1]{IZ95}, 
which asserts that the completion of $\Q_p(X)$ with respect to \emph{any} 
residually algebraic torsion extension of $v$ is isomorphic to a complete field between $\Q_p$ and $\C_p$ --- which 
in Example \ref{prop:notValpha} is the same as between $\Q_p$ and  $\C_p$, the field $\Q_p$ being complete. 
The isomorphism there is  topological, as the proof of  \cite[Proposition 1]{IZ95} makes explicit, 
so that the two statements do bear on the same thing. The extension of Example \ref{prop:notValpha} 
is torsion and residually algebraic, so the proposition does apply to it. 
The proof given there passes from a sequence $\{(a_n,\delta_n)\}$ of minimal pairs defining $w$ to the element
$x=\lim_na_n$ of $\C_p$, and so needs $\{a_n\}$ to be Cauchy, that is, the $\delta_n$ to be
unbounded; what the construction of \cite{APZAll} provides is only that they increase without
attaining a limit in $\Gamma_{\overline v}\cong\Q$, which leaves them possibly bounded. Under
the additional hypothesis that they be unbounded the conclusion is correct, 
and by the proof of Theorem \ref{thm:main} discreteness of $W$ is one way of securing it; \cite[Theorem 3]{IZ95}, 
whose sequence satisfies $v(a_{n+1}-a_n)\geq n$, is unaffected. 
No standing hypothesis of \cite{IZ95} excludes Example \ref{prop:notValpha}: 
the field $\Q_p$ is perfect, carries a valuation of rank one, and has $\oK$ denumerably generated over $\K$.} \eoe
\end{remark}

It remains to characterize the $\alpha\in\bK$ for which $V_{\alpha}$ is a DVR.
To this end, we consider the elements of $\bK$ of bounded ramification, namely:
$$\bK^{\br}=\{\alpha\in\bK\mid e_{\alpha}<\infty\}.$$
Clearly, $\oK\subseteq\bK^{\br}$, by the remarks in Section \ref{sec:eps}.
However, this containment is strict, by Proposition \ref{prop:nonvacuous}, and,
by Example \ref{ex:nondiscrete}, some condition is needed.
 The following proposition shows that the elements of $\bK^{\br}$ which are transcendental
 over $K$ are precisely those for which $V_{\alpha}$ is a DVR and what governs discreteness is
the ramification of the algebraic elements approximating $\alpha$.
It also settles a point left open by Example \ref{ex:nondiscrete}:
discreteness is characterized there by the existence of \emph{some} approximating sequence with bounded ramification indices,
so that, as $V_{\alpha}$ is not a DVR in that example, no such sequence exists for that $\alpha$
at all, and the unboundedness of $e_{s_N}=\ell^N$ computed there is not an accident of the sequence chosen.
We mention that for $V=\Z_{(p)}$ the set of $\alpha\in\C_p$ for which $\Z_{(p),\alpha}$ is a DVR is denoted
by $\C_p^{\text{br}}$ in \cite[Remark 2.18 \& Corollary 2.28]{p25}. The following proposition generalizes
\cite[Proposition 2.20]{p25}.

\begin{proposition}\label{prop:criterion}
Let $V$ be a DVR with perfect residue field and let $\alpha\in\bK$ be transcendental over $K$.
The following { conditions } are equivalent.
\begin{enumerate}
\item $\alpha\in\bK^{\br}$.
\item $V_{\alpha}$ is a DVR.
\item There exists a Cauchy sequence $\{s_n\}_{n\in\N}\subset\oK$ converging to $\alpha$
and with $\{e_{s_n}\}_{n\in\N}$ bounded.
\end{enumerate}
The implication $(3)\Rightarrow(2)$ holds without the hypothesis that the residue
field of $V$ be perfect.
\end{proposition}

\begin{proof}
Let $\alpha\in\bK$ be transcendental over $K$. In this case the evaluation map $\phi(X)\mapsto\phi(\alpha)$ is 
an isomorphism from $K(X)$ to $K(\alpha)$ and $V_{\alpha}$ is by definition the pullback of $\bO\cap K(\alpha)$ 
under this evaluation map. In particular, $V_{\alpha}$ is a DVR if and only if $\bO\cap K(\alpha)$ is a DVR. 
Since $K(\alpha)\subseteq\K(\alpha)$ is an immediate extension (for example because $\K(\alpha)$ 
is contained in the completion of $K(\alpha)$), it follows that $\bO\cap K(\alpha)$ is a DVR 
if and only if $\bO\cap \K(\alpha)$ is a DVR. Thus, we have shown:
\begin{equation}\label{equivalence DVRs}
V_{\alpha}\text{ is a DVR}\iff \bO\cap\K(\alpha)\text{ is a DVR}\iff \bO\cap K(\alpha)\text{ is a DVR}
\end{equation} 

$(1)\Rightarrow(2)$ Suppose now (1) holds, that $\alpha\in \bK^{\br}$, that is, $e_{\alpha}<\infty$ 
 and $\alpha$ is transcendental over $K$. 
 This is equivalent to saying that $\bO\cap\K(\alpha)$ is a DVR, since $\Gamma_v$ has finite index in
 the value group of this valuation domain by assumption. 
 By \eqref{equivalence DVRs}, $V_{\alpha}$ is a DVR, so we have (2).

$(2)\Rightarrow(3)$. By Remark \ref{rem:Valpha} the extension $V_{\alpha}$ of $V$
is residually algebraic, and it is a DVR by hypothesis, so Theorem
\ref{thm:main} applies to $W=V_{\alpha}$. If $[k_w:k_v]<\infty$ then
$\alpha\in\oK$ by that theorem, and the constant sequence $s_n=\alpha$ satisfies
 (3), with the convention $v(0)=\infty$. Assume $[k_w:k_v]=\infty$ and put $e=[\Gamma_w:\Gamma_v]$. By assertion
(3) of Theorem \ref{thm:main} there are $\alpha'\in\bK$ with
$V_{\alpha'}=W=V_{\alpha}$ and $s_n'\in\oK$ with $\alpha'=\lim\limits_ns_n'$ and
$e(\K(s_n')\mid\K)=e$ for all large $n$. By Remark
\ref{rem:uniqueness} there is $\sigma\in G$ with $\alpha'=\sigma(\alpha)$. Put
$s_n=\sigma^{-1}(s_n')\in\oK$. Since $\sigma$ is an isometry of $\oK$ fixing $K$
pointwise, extended by continuity to $\bK$,
$$v(\alpha-s_n)=v\big(\sigma(\alpha)-s_n'\big)=v(\alpha'-s_n')\to\infty,$$
while $\sigma^{-1}$ carries $K(s_n')$ onto $K(s_n)$ preserving $v$, so that
$$e_{s_n}=e_{s_n'}\leq e$$
for all large $n$. Enlarging the bound to
accommodate the finitely many remaining indices, the sequence $\{s_n\}$ satisfies
(3).

$(3)\Rightarrow(1)$. If $\alpha\in\oK$ then clearly $\alpha\in\bK^{\br}$. 
Suppose that $\alpha\not\in\oK$, so, in particular, $\alpha$ is 
transcendental over $\K$. Let $M$ bound the $e_{s_n}$, so that
$v(\K(s_n)^{\times})=\frac{1}{e_{s_n}}\Gamma_v\subseteq\frac{1}{M!}\Gamma_v$
for
every $n$. Let $f\in \K[X]$ be non-zero. Then $f(\alpha)\neq0$, $\alpha$ being
transcendental over $\K$, and $f(s_n)\to f(\alpha)$, so that
$v(f(\alpha))=v(f(s_n))\in\frac{1}{M!}\Gamma_v$ for $n$ large. Hence the value
group $v(\K(\alpha)^{\times})$ of $V_{\alpha}$ is a subgroup of
$\frac{1}{M!}\Gamma_v$ containing $\Gamma_v$, therefore infinite cyclic, 
so $e_{\alpha}<\infty$ and $\alpha\in\bK^{\br}$.
 Perfectness of the residue field has not been used.
\end{proof}

\begin{remark}\label{que:discrete}
{\em Proposition \ref{prop:criterion} is a criterion bearing on $\alpha$ alone, but it is
expressed through ramification indices, and one may ask whether the degrees suffice: is the
discreteness of $V_{\alpha}$ detected by the degrees $[K(s):K]$ and the values
$v(\alpha-s)$, for $s\in\oK$, without reference to the ramification indices? The naive form
of this question has a negative answer. In Example \ref{ex:nonvacuous} the sequence
exhibited there has $e_{s_N}=1$, whereas there is no sequence $t_n\in\oK$ at all
with $v(\alpha-t_n)\to\infty$ and $[K(t_n):K]$ bounded, since $\K$ is locally compact and
such a sequence would force $\alpha$ to be algebraic over $\K$; so the degrees cannot simply
be substituted for the ramification indices in condition $(3)$. Example \ref{ex:nondiscrete}
gives no evidence either way: there $K(s_N)=K(\theta_N)$ is totally ramified over $K$, so
that $[K(s_N):K]=e_{s_N}=\ell^N$ and the two conditions coincide. Whether some
other condition on the function $s\mapsto\big([K(s):K],v(\alpha-s)\big)$ would do, we do not
know.} \eoe
\end{remark}

Theorem \ref{thm:main} and Proposition \ref{prop:criterion} now combine 
into a single description of the extensions of $V$ to $K(X)$ which are DVRs and residually algebraic over $V$; 
this is the form in which the two results are most conveniently used, 
and the one announced in the introduction. Before stating it we record how
large the class so described is. Both cases of Theorem \ref{thm:main} occur, and one can say
exactly when each of them does: the case $\alpha\in\oK$ arises precisely when $K$ is not
complete, by Proposition \ref{rem:byproduct}, and the case $\alpha\notin\oK$ precisely when
$[\ok:k_v]=\infty$, by Proposition \ref{prop:nonvacuous}. In particular the class is empty
altogether when $K$ is complete and $k_v$ is algebraically closed or real closed, by Corollary
\ref{cor:vacuous}.

\begin{corollary}\label{cor:description}
Let $V$ be a DVR with perfect residue field. The extensions of $V$ to $K(X)$ which
are DVRs and residually algebraic over $V$ are exactly the valuation domains
$V_{\alpha}$, where
$\alpha$ runs over the elements of $\bK$ transcendental over $K$ for which there
is a sequence $\{s_n\}\subset\oK$ with $v(\alpha-s_n)\to\infty$ and
$\{e_{s_n}\}$ bounded; and
$\alpha\in\oK$ precisely when the residue field extension is finite.
\end{corollary}

\begin{proof}
Combine Theorem \ref{thm:main}, Remark \ref{rem:Valpha} and
Proposition \ref{prop:criterion}.
\end{proof}

We prove now that the set $\bK^{\br}$ is actually a \emph{subfield} of $\bK$. 
For the next result, we consider $\K^{\ur}$, the maximal unramified extension of $\K$ in $\oK$ 
(which is the compositum of all the unramified extensions of $\K$ in $\oK$) and $\bK^{\ur}$ 
its completion, which is a subfield of $\bK$. That is,
$\bK^{\ur}=\widehat{\K^{\ur}}$; the completion of the maximal unramified
extension of a finite extension $F$ of $\K$ is written $\widehat{F^{\ur}}$, as
in Theorem \ref{thm:union}.  Note that the value groups of $\K,\K^{\ur}$ and $\bK^{\ur}$ 
are the same and they are equal to $\Gamma_v$ (which is isomorphic to $\Z$).

\begin{theorem}\label{bounded ramification subfield}
Let $V$ be a DVR with perfect residue field. Then 
the set $\bK^{\br}$ is equal to the algebraic closure of $\bK^{\ur}$ in $\bK$. 
More precisely, for any $\alpha\in\bK^{\br}$, the extension $\bK^{\ur}\subseteq\bK^{\ur}(\alpha)$ is totally ramified.  
In particular, $\bK^{\br}$ is a field. 
\end{theorem}

\begin{proof}
Let $\alpha\in\bK$ be algebraic over $\bK^{\ur}$. 
Then, the extension $\bK^{\ur}\subseteq \bK^{\ur}(\alpha)$ is  finite, so in particular, 
the value group of $\bK^{\ur}(\alpha)$ is discrete. Since $\K(\alpha)\subseteq \bK^{\ur}(\alpha)$ 
it follows that the value group of $\K(\alpha)$ is discrete, that is, $e_{\alpha}<\infty$. 
By definition,  $\alpha\in\bK^{\br}.$

Conversely, let $\alpha\in\bK^{\br}$. We consider the field $F_{\alpha}=\widehat{\K(\alpha)}\cap\oK$ 
and its subfield $F_{\alpha}\cap \K^{\ur}$, which is the maximal unramified extension of $\K$ in $F_{\alpha}$. 
In particular, since $k_v$ is perfect, the residue fields of $F_\alpha$ and of
$F_\alpha \cap \widehat{K}^{\,\mathrm{ur}}$ are the same: every element of the
former is separable over $k_v$ and lifts, by Hensel's lemma as in Step~1 of
Lemma~\ref{lem:ram}, to an unramified subextension of
$F_\alpha \mid \widehat{K}$; see also \cite[Ch.~II, \S7]{Neu} and
\cite[Ch.~III, \S5]{Ser}. By \cite[Theorem 1]{IZ95}, $\widehat{F_{\alpha}}=\widehat{\K(\alpha)}$, 
and the value group of the right-hand side of the last equality is discrete  by assumption. 
Moreover, since, by the previous paragraph, $\widehat{F_{\alpha}}$ and 
$\widehat{F_{\alpha}\cap \K^{\ur}}$ have the same residue field, 
and the respective valuation domains are complete DVRs, by \cite[Lemma 2.4]{PerTransc}, 
the extension $\widehat{F_{\alpha}\cap \K^{\ur}}\subseteq \widehat{F_{\alpha}}$ is finite. 
In particular, $\alpha$ is algebraic over $\widehat{F_{\alpha}\cap \K^{\ur}}$, 
so, a fortiori, it is algebraic over $\bK^{\ur}$, as claimed.

It is well-known that the residue field of $\K^{\ur}$ is equal to  $k_v^{\mathrm{sep}}$ 
(see, for example, Step 1 of the proof of Lemma \ref{lem:ram}), 
hence to $\overline{k_v}$, an algebraic closure of the residue field of $\K$, the field $k_v$ being perfect. 
In particular, the residue field of $\bK^{\ur}$ is equal to $\overline{k_v}$. 
Since we have just shown that for every $\alpha\in\bK^{\br}$ the extension $\bK^{\ur}\subseteq\bK^{\ur}(\alpha)$ is finite 
and the residue field of the bottom field is the algebraically closed field $\overline{k_v}$, 
it follows that $\bK^{\ur}(\alpha)$ has residue field equal to $\overline{k_v}$. 
Thus, the extension $\bK^{\ur}\subseteq\bK^{\ur}(\alpha)$ is totally ramified.

The last claim is now straightforward.
\end{proof}

The following is an alternative and useful characterization of the field $\bK^{\br}$, 
viewed as the union of a directed family of the completions of 
the maximal unramified extensions of all the finite extensions of $\K$. As before, given a finite extension $F$ of $\K$, 
we denote by $F^{\ur}$ the maximal unramified extension of $F$ in $\oK$.

\begin{theorem}\label{thm:union}
Let $V$ be a DVR with perfect residue field.  Let $\mathcal{K}$ be the family of all finite extensions of $\K$. Then
$$\bK^{\br}=\bigcup_{F\in\mathcal{K}}\widehat{F^{\ur}} .$$
\end{theorem}

\begin{proof}
Let $F \in \mathcal{K}$ and $\alpha\in\widehat{F^{\ur}}$. 
By definition, there exists a Cauchy sequence $\{s_n\}_{n\in\N}\subset F^{\ur}$ converging to $\alpha$. 
Considering the tower of extensions $\K\subseteq F\subset F^{\ur}\subset\oK$, 
we see that the ramification index $e(F^{\ur}\mid \K)$ is equal to $e(F\mid \K)$. 
Hence, $\{e_{s_n}\}_{n\in\N}$ is bounded by $e(F\mid \K)$, so by Proposition \ref{prop:criterion}, $\alpha\in\bK^{\br}$.


Conversely, let $\alpha\in \bK^{br}$. By Theorem \ref{bounded ramification subfield}, $\alpha$
is algebraic over $\bK^{ur}$. Assume first that $\alpha$ is separable over $\bK^{ur}$
and denote by $g\in \bK^{ur}[X]$ its minimal polynomial. If $f\in\widehat{K}^{ur}[X]$
is a polynomial of the same degree as $g$ and sufficiently close to $g$, then $f$
is irreducible and $\bK^{ur}(\alpha)= \bK^{ur}(\alpha')$, where $\alpha'$ is a root of
$f(X)$ (\cite[Ch.~XII, Proposition 2.5]{Lang}, a consequence of Krasner's lemma).  By
construction, $\alpha'$ is algebraic over $\widehat{K}^{ur}$, so it is in
$\overline{\widehat{K}}$.   Put $F=\widehat{K}(\alpha')\in\mathcal{K}$. Then
$\widehat{K}^{ur}(\alpha')=\widehat{K}^{ur}\cdot F=F^{ur}$: the extension
$\widehat{K}^{ur}\mid\widehat{K}$ is unramified, hence so is
$\widehat{K}^{ur}\cdot F\mid F$ by base change (Step 2 of the proof of
Lemma~\ref{lem:ram}), which gives $\widehat{K}^{ur}\cdot F\subseteq F^{ur}$; 
conversely the residue field of a finite unramified extension of $F$ is a finite
extension of $k_{v}$, separable since $k_{v}$ is perfect, so it is the residue
field of a finite unramified extension of $\widehat{K}$ by Step~1 of the same
proof, whence $F^{ur}\subseteq\widehat{K}^{ur}\cdot F$.
Finally, $\bK^{ur}(\alpha')$
is a finite extension of the complete field $\bK^{ur}$, hence complete, and
$F^{ur}=\widehat{K}^{ur}(\alpha')$ is dense in $\bK^{ur}(\alpha')$, because $\widehat{K}^{ur}$ is dense
in $\bK^{ur}$; therefore
$\alpha \in \bK^{ur}(\alpha) = \bK^{ur}(\alpha') = \widehat{F^{ur}}$.

In general, let $p^m$ be the inseparability degree of $\alpha$ over $\bK^{ur}$, so
that $\gamma:=\alpha^{p^m}$ is separable over $\bK^{ur}$, and $\gamma\in
\bK^{br}$, the value group of $\bK^{ur}(\gamma)$ being contained in the
discrete value group of $\bK^{ur}(\alpha)$. By the case just treated there is
$F_1\in\mathcal{K}$ with $\gamma\in\widehat{F_1^{ur}}$; we may assume $m>0$, so
that $\operatorname{char}\widehat{K}=p$. As in Example~\ref{ex:families},
$\widehat{F_1^{ur}}=\overline{k_v}((\varpi))$ for a uniformizer $\varpi$ of $F_1$,
the residue field $\overline{k_v}$ being perfect. Put
$F=F_1(\varpi^{1/p^m})\in\mathcal{K}$. Then
$\widehat{F^{ur}}=\overline{k_v}((\varpi^{1/p^m}))$ contains a $p^m$-th root of
every element of $\overline{k_v}((\varpi))$, again because $\overline{k_v}$ is
perfect; hence $\alpha=\gamma^{1/p^m}\in\widehat{F^{ur}}$.
\end{proof}

Neither Theorem \ref{bounded ramification subfield} nor Corollary
\ref{cor:description} says what $\bK^{\br}$ is in a concrete instance.  
For the two families of Example \ref{ex:families} it can be named outright, 
in characteristic zero, and Corollary \ref{cor:description} becomes a statement about Puiseux series.

\begin{example}\label{ex:puiseux}
Keep the notation of Example \ref{ex:families} and let
$\operatorname{char}F=0$.  Then $\bK^{\br}$ is the field of Puiseux series over
$\overline\kappa$:
$$\bK^{\br}=\bigcup_{N\geq1}\overline{\kappa}((z^{1/N})),\qquad\text{while}\qquad
\oK=\bigcup_{N\geq1}\bigcup_{\kappa'}\kappa'((z^{1/N})),$$
the inner union being over the finite extensions $\kappa'\mid\kappa$.
 Consequently the extensions of $V$ to $K(X)$ which are DVRs and residually
algebraic over $V$ are exactly the $V_{\alpha}$ for the Puiseux series
$$\alpha=\sum_{n\geq n_0}a_nz^{\,n/N},\qquad a_n\in\overline\kappa, \ \ N\in\N \text{ is fixed }$$
that are transcendental over $K$; the index $[\Gamma_w:\Gamma_v]$ divides $N$, 
and $[k_w:k_v]$ is finite precisely when the $a_n$'s generate a finite extension of $\kappa$. 
Example \ref{ex:CC} is the case $\kappa=\overline\kappa=\C$ of this
description and Example \ref{ex:Qps} is the case $\kappa=\Q$, $N=1$.  The
hypothesis on the characteristic cannot be dropped: in characteristic $p$, the
algebraic closure of $\overline\kappa((z))$ strictly contains the Puiseux field,
by an example of Chevalley \cite[\S2]{a56} (or see \cite[p. 64]{c51}),
and already the roots of $Y^p-Y=z^{-1}$ lie outside it.
\end{example}

\begin{proof}
The field $\bK$ is algebraically closed, being the completion of the
algebraically closed field $\oK$ of rank one, so Theorem
\ref{bounded ramification subfield} implies that $\bK^{\br}$ is an algebraic closure of
$\bK^{\ur}$, the last field being  equal to $\overline\kappa((z))$ by Example \ref{ex:families}. By the
theorem of Newton and Puiseux (\cite[Ch. IV, \S 2, Proposition 8]{Ser}),  
the algebraic closure of $\overline\kappa((z))$ is equal to  the displayed Puiseux field, 
since the residue field $\overline\kappa$ of $\overline\kappa((z))$ is  algebraically closed.   
The description of $\oK$ follows from the same theorem applied to
$\widehat K\cong\kappa((z))$, once one knows that every finite extension $E$ 
of $F=\kappa((z))$ is contained in $\kappa'((z^{1/N}))$ for a suitable finite extension 
$\kappa'\,|\,\kappa$ and a suitable $N\ge 1$; the containment cannot in general be
replaced by an equality, as the example $E=\mathbb{Q}((z))(\sqrt{2z})$, whose residue
field is $\mathbb{Q}$, already shows. To see it, put $N=e(E\,|\,F)$, 
denote by $\mathcal O_E$  the valuation ring of $E$, and let $\kappa'_0$ be the residue field of $E$, 
a finite and hence separable extension of $\kappa$.
Lifting a primitive element of $\kappa'_0$ by Hensel's lemma, as in Step 1 of the
proof of Lemma \ref{lem:ram}, we obtain a copy of $\kappa'_0$ inside $\mathcal O_E$, so that
$\kappa'_0((z))\subseteq E$ and $E\,|\,\kappa'_0((z))$ is totally ramified of degree
$N$. Let $\pi$ be a uniformizer of $E$ and write $z=u\pi^N$ with
$u\in\mathcal O_E^{\times}$; let $\bar u\in\kappa'_0$ have the same residue as $u$.
The unit $u/\bar u$ has residue $1$, and $N$ is invertible, $\kappa$ having
characteristic zero, so Hensel's lemma provides an $N$-th root of $u/\bar u$ in $E$.
Setting $\kappa'=\kappa'_0(\bar u^{1/N})$ and
$\rho=\bar u^{1/N}(u/\bar u)^{1/N}\pi\in E\kappa'$ we get $\rho^N=z$, whence
$E\subseteq E\kappa'=\kappa'((z))(\rho)\subseteq\kappa'((z^{1/N}))$. Conversely, each
$\kappa'((z^{1/N}))$ is a finite extension of $F$, of degree $[\kappa':\kappa]\,N$,
and the description of $\overline{\widehat K}$ follows.  

The description of the extensions announced in the statement now follows from Corollary \ref{cor:description},
$\mathbb{K}^{\mathrm{br}}$ having just been identified with the Puiseux field;
an $\alpha$ as displayed lies in $\overline\kappa((z^{1/N}))$, whose value group is
$\frac1N\Gamma_v$, whence the assertion on the index, and it lies in $\oK$
precisely when its coefficients lie in a single finite extension of $\kappa$,
which by Theorem \ref{thm:main}\,(2) is the condition on $[k_w:k_v]$.
\end{proof}

The one hypothesis of Corollary \ref{cor:description} that we have not been able to
dispense with is the perfectness of the residue field, and we close the paper with the
following question it raises.

\begin{question}\label{que:imperfect}
Does Theorem \ref{thm:main} hold when the residue field of $V$ is not perfect?
\end{question}

We can at least say where the hypothesis is used. In the proof of Theorem
\ref{thm:main} it serves only to guarantee the hypothesis of Lemma \ref{lem:ram}, that
is, the separability of the residue field extensions of the finite extensions
$K(s)\mid K$, $s\in\oK$; by Example \ref{rem:sharp} that hypothesis cannot be
removed from Lemma \ref{lem:ram} itself, so the present strategy of proof does not
extend, but we know of no counterexample to the statement of the theorem.
The smallest case to test is $V=\kappa_0[[t]]$ with $\kappa_0=\mathbb{F}_p(u)$, the
field of Example \ref{rem:sharp} itself.

The same question has a second side. Proposition \ref{prop:nonvacuous} uses
perfectness independently of
Lemma \ref{lem:ram}, so an affirmative answer would not remove the hypothesis by itself
 from the results of this paper; and there is more lost. For $k_v$ imperfect, the
extension $\ok\mid k_v$ is inseparable, so that the unramified extensions on which the
construction rests are no longer available and the implication $(1)\Rightarrow(2)$ of
Proposition \ref{prop:nonvacuous} breaks down. Only the implication $(3)\Rightarrow(1)$,
which uses no separability, survives, and we do not know whether over such a $V$ there is
any $W$ as in Theorem \ref{thm:main} with $[k_w:k_v]=\infty$.

\subsection*{Acknowledgments}
The first author was supported by Basic Science Research Program through the National Research Foundation of Korea (NRF)
funded by the Ministry of Education (RS-2017- NR027863). 
The second author  is a member of the National Group for Algebraic and Geometric Structures and their Applications 
(GNSAGA) of the Italian Mathematics Research Institute (INdAM).

\subsection*{Declaration of generative AI and AI-assisted technologies in the manuscript preparation process}

 During the preparation of this work the authors used Claude, an AI assistant developed by Anthropic, in order to explore different mathematical contents about the topics of this paper. After using this tool/service, the authors reviewed and edited the content as needed and take full responsibility for the content of the published article.

\end{document}